\documentclass[12pt]{amsproc}

\usepackage[a4paper, margin=1in]{geometry}
\usepackage{amsmath,amssymb}
\usepackage{bbm}
\usepackage{tikz}
\usepackage{hyperref}

\numberwithin{equation}{section}

\newtheorem{thm}{Theorem}[section]%
\newtheorem{lem}[thm]{Lemma}%
\newtheorem{cor}[thm]{Corollary}%
\newtheorem{defi}[thm]{Definition}%
\newtheorem{pro}[thm]{Proposition}%
\newtheorem{rem}[thm]{Remark}%

\newtheorem*{thm A}{Theorem A}
\newtheorem*{thm B}{Theorem B}

\begin{document}

	\title[Some classifications of finite-dimensional Hopf algebras over $H_{b:x^2y}$]{Some classifications of finite-dimensional Hopf algebras over the Hopf algebra $H_{b:x^2y}$ of Kashina }

	\author[Y.H. Wu]{Yihan Wu}
	\address{School of Mathematical Sciences, MOE Key Laboratory of Mathematics and Engineering Applications \& Shanghai Key Laboratory of PMMP, East China Normal University, Shanghai 200241, China}
	\address{Department of Mathematics and Statistics, University of Hasselt, Diepenbeek 3590, Belgium}
	\email{52265500002@stu.ecnu.edu.cn}
	
	\author[H.Y. Wang]{Hengyi Wang}
	\address{School of Mathematical Sciences, MOE Key Laboratory of Mathematics and Engineering Applications \& Shanghai Key Laboratory of PMMP, East China Normal University, Shanghai 200241, China}
	\address{UFR Mathématique de Université Paris Cité, CNRS, IMJ-PRG, Bâtiment Sophie Germain – 8 place Aurélie Nemours, 75013 Paris, France}
	\email{52265500001@stu.ecnu.edu.cn}
	
	\author[N.H. Hu]{Naihong Hu$^*$}
	\address{School of Mathematical Sciences, MOE Key Laboratory of Mathematics and Engineering Applications \& Shanghai Key Laboratory of PMMP, East China Normal University, Shanghai 200241, China}
	\email{nhhu@math.ecnu.edu.cn}
	
	\thanks{
		This work is supported by the NNSF of China (Grant No. 12171155), and in part by the Science and Technology Commission of Shanghai Municipality (Grant No. 22DZ2229014).
	}
	
	\date{}
	
	\maketitle
	
	\begin{abstract}
		Let $H$ be the $16$-dimensional nontrivial
		(namely, noncommutative and noncocommutative) semisimple Hopf
		algebra $H_{b:x^2y}$ classified by  Kashina. We figure out
		all simple Yetter-Drinfeld $H$-modules,  and then determine all
		finite-dimensional Nichols algebras satisfying the constraint condition $\mathcal{B}(V)\cong
		\bigotimes_{i\in I}\mathcal{B}(V_i)$, where $V=\bigoplus_{i\in
			I}V_i$, each $V_i$ is a simple object in $_H^H\mathcal{YD}$.
		Finally, we describe some  liftings of  the corresponding  Radford biproducts
		$\mathcal{B}(V)\sharp H$, which provide some classifications of finite dimensional Hopf algebras with $H$ as their coradical.
	\end{abstract}
	
		 \normalsize
	\section{Introduction}	
	\noindent\textbf{1.1.}
	Let $\mathbb{K}$ be an algebraically closed field of characteristic
	zero.In 1975, Kaplansky \cite{K} proposed ten conjectures, one of which concerns the number of isomorphism classes of finite-dimensional Hopf algebras over $\mathbb{K}$.
	So far, although numerous classification results have been obtained, there is still no unified standard method. A notable exception is the lifting method proposed by Andruskiewitsch and Schneider \cite{AS} in 1998, which has since played a central role in the classification of finite-dimensional (co-)pointed Hopf algebras. (eg. \cite{AS2}, \cite{AS07},	\cite{AHS}, \cite{AS3},	\cite{AV},  \cite{AI}, \cite{FG},	\cite{H06} etc.).
	
	Actually, the lifting method works well for those classes of
	finite-dimensional Hopf algebras with Radford biproduct structure
	or their deformations by certain Hopf $2$-cocycles. Let $A$ be a
	Hopf algebra. If the coradical $A_0$ is a Hopf subalgebra, then the
	coradical filtration $\{A_n\}_{n\geq 0}$ is a Hopf algebra
	filtration, and the associated graded coalgebra
	$grA=\bigoplus_{n\geq 0}A_n/A_{n-1}$ is also a Hopf algebra, where
	$A_{-1}=0$. Let $\pi:gr A\rightarrow A_0$ be the homogeneous
	projection. By a theorem of Radford~\cite{R85}, there exists a
	unique connected graded braided Hopf algebra $R=\bigoplus_{n\geq
		0}R(n)$ in the monoidal  braided category $^{A_0}_{A_0}\mathcal{YD}$ such
	that $grA\cong R\sharp A_0$. We call $R$ or $R(1)$ the diagram or
	infinitesimal braiding of $A$, respectively. Moreover, $R$ is
	strictly graded, that is, $R(0)=\mathbb{K}, ~R(1)=\mathcal{P}(R)$
	(see Definition 2.1 \cite{AS2}).
	
	\smallskip\noindent\textbf{1.2.}
	The lifting method has been applied to classify some
	finite-dimensional pointed Hopf algebras such as
	\cite{AFG10},  \cite{AFG11}, \cite{ACG15a},\cite{ACG15b},  \cite{FG}, \cite{GGI},
	etc., and copointed Hopf algebras \cite{AV},~\cite{GIV}, etc.
	Nevertheless, there are a few classification results on
	finite-dimensional Hopf algebras whose coradical is neither a group
	algebra nor the dual of a group algebra, for instance, \cite{AGM15},
	\cite{CS}, \cite{GG16}, \cite{HX16}, \cite{HX18}, \cite{S16},
	\cite{X19},	\cite{X192}, \cite{X21}, \cite{X23}	\cite{X232}, \cite{X24}  etc. In fact, Shi began a program in \cite{S16} to
	classify the objects of finite-dimensional growth from a given
	nontrivial $8$-dimensional semisimple Hopf algebra $A_0=H_8$ (Kac-Paljutkin algebra) via some relevant
	Nichols algebras $\mathcal B(V)$ derived from its semisimple
	Yetter-Drinfeld modules $V\in {^{A_0}_{A_0}\mathcal{YD}}$. On the other hand, Kashina  \cite {K00} classified the $16$-dimensional semisimple Hopf algebras, which correspond to $16$ isoclasses, i.e.,
	\begin{equation*}
		\begin{split}
		H_{a:1},\ H_{a:y},\ H_{b:1},\ (H_{b:1})^*,\ H_{b:y},\ H_{b:x^2y},\ H_{B:1},\ H_{B:X}, \\
		H_{c:\sigma_0},\ H_{c:\sigma_1},\ (H_{c:\sigma_1})^*,\ H_{C:1},\ H_{C:\sigma_1},\ H_{d:-1,1},\ H_{d:1,1},\ H_E.
		\end{split}
	\end{equation*}
	Subsequently, Zheng, Gao, Hu, et al. carried out  similar classifications of finite dimensional Hopf algebras with the  coradicals arising from  some of  $16$-dimensional Kashina semisimple Hopf algebras: $H_{b:1}$ (\cite{Z211}), $H_{d:-1,1}$ (\cite{Z212}), $(H_{b:1})^*$ (\cite{Z23}), $H_{a:y}$ (\cite{Z25}) respectively. This paper  aims to classify some finite dimensional Hopf algebras over the Kashina semisimple Hopf algebra $H_{b:x^2y}$.

	The classification of finite-dimensional Hopf algebras $A$ whose
	coradical $A_0$ is a Hopf subalgebra proceeds through the following main  three steps:
	
	\begin{enumerate}
		\item[(a)] Determine those Yetter-Drinfeld simple or semisimple $A_0$-modules $V$ in $^{A_0}_{A_0}\mathcal{YD}$ such that the corresponding Nichols algebra $\mathcal{B}(V)$ is finite-dimensional and present $\mathcal{B}(V)$ by generators and relations.
		
		\item[(b)] If $R=\bigoplus_{n\geq 0}R(n)$ is a finite-dimensional Hopf algebra in $^{A_0}_{A_0}\mathcal{YD}$ with $V=R(1)$, decide if $R\cong \mathcal{B}(V)$.
		
		\item[(c)] Given $V$ as in (a), classify all finite dimensional Hopf algebras $A$ such that $\operatorname{gr}A \cong \mathcal{B}(V)\sharp A_0$. We call $A$ a lifting of $\mathcal{B}(V)$ over $A_0$.
	\end{enumerate}
	
	\smallskip\noindent\textbf{1.3.}
	In this paper, we first fix a $16$-dimensional non-trivial semisimple Hopf
	algebra $H=H_{b:x^2y}$,  and then study the above
	questions. For the definition of Hopf algebra $H$, see Definition \ref{H}. We first determine the Drinfeld double $D:=D(H^{cop})$
	of $H^{cop}$ and describe the simple $D$-modules. In fact, we prove
	in Theorem \ref{sims} that there are $32$ one-dimensional simple objects
	$\mathbb{K}_{\chi_{i,j,k,l}}$ with $0\leq i,~k,~l<2, ~0\leq j<4$,
	and $56$ two-dimensional simple objects $V_{i,j,k,l}$ with
	$(i,j,k,l)\in \Omega$, $W_{i,j,k}$ with $(i,j,k)\in
	\Lambda^1$, $U_{i,j,k,l}$ with $(i,j,k,l)\in \Lambda^2$. Using the
	fact that the braided categories $_D\mathcal{M}$ and
	$_H^H\mathcal{YD}$ are monoidally isomorphic  each other (see \cite{CK}), we
	actually get the simple objects in $_H^H\mathcal{YD}$. Furthermore, we get
	all the possible finite-dimensional Nichols algebras of semisimple
	objects  satisfying  the constraint condition $\mathcal{B}(N)\cong \bigotimes_{i\in
		I}\mathcal{B}(N_i)$,  for $N=\oplus_{i\in I} \, N_i$. The following  is our first main result.
	
	\begin{thm A}
		Suppose $N=\oplus_{i\in I}N_{i}$ is semisimple with each simple object $N_i$ in $_H^H\mathcal{YD}$ such that the Nichols algebra $\mathcal {B}(N)\cong \otimes  _{i\in I}\mathcal {B}(N_i)$. Then dim $\mathcal{B}(N)< \infty$ if and only if N is isomorphism to one of the following Yetter-Drinfel modules:\\
	$(1)~\Omega_{i,j}=V_i\oplus V_j,~~i,j\in \mathbb{I}_{1,8};$\\
	$(2)~ \Omega_{1,i,j}(n_1,n_2,n_3)=M_{1}^{\oplus n_1}\oplus V_{2i-1}^{\oplus n_2}\oplus V_{2j}^{\oplus n_3}, \, i, \, j\in \mathbb{I}_{1,2}, \, n_1\geq 1, \, n_2,\, n_3\in \mathbb{I}_{0,1}, \, (n_1,n_2,n_3)\neq (2,0,0);$\\
	$(3)~\Omega_{7,i,j}(n_1,n_2,n_3)=M_7^{\oplus n_1}\oplus V_{2i-1}^{\oplus n_2}\oplus V_{2j}^{\oplus n_3}, \, i, \, j \in  \mathbb{I}_{1,2}, \, n_1\geq 1,\, n_2,\, n_3\in \mathbb{I}_{0,1},\, (n_1,n_2,n_3)\neq (2,0,0);$\\
	$(4)~\Omega_{8,i,j}(n_1,n_2,n_3)=M_8^{\oplus n_1}\oplus V_{2i-1}^{\oplus n_2}\oplus V_{2j}^{\oplus n_3}, \, i\in \mathbb{I}_{3,4}, \, n_1\geq 1, \,n_2, \,n_3\in \mathbb{I}_{0,1}$, $(n_1,n_2,n_3)\neq (2,0,0);$\\
	$(5)~\Omega_{9,i,j}(n_1,n_2,n_3)=M_9^{\oplus n_1}\oplus V_{2i-1}^{\oplus n_2}\oplus V_{2j}^{\oplus n_3},\, i\in \mathbb{I}_{3,4},\, j\in  \mathbb{I}_{1,2}, \,  n_1\geq 1, \, n_2,n_3\in \mathbb{I}_{0,1}$, $(n_1,n_2,n_3)\neq (2,0,0);$\\
	$(6)~\Omega_{10,i,j}(n_1,n_2,n_3)=M_{10}^{\oplus n_1}\oplus V_{2i-1}^{\oplus n_2}\oplus V_{2j}^{\oplus n_3}, \, i\in \mathbb{I}_{3,4},\, j\in \mathbb{I}_{1,2}, \, n_1\geq 1, \, n_2,n_3\in \mathbb{I}_{0,1}$, $(n_1,n_2,n_3)\neq (2,0,0);$\\
	$(7)~\Omega_{i}=M_i\oplus M_{i}, \, i\in \mathbb I_{1,12};$\\
	$(8)~{\Omega_{i}}^{(1)}=M_1\oplus M_{i}$, $i\in \{6,8\};$\\
	$(9)~{\Omega_{i}}^{(2)}=M_{i}\oplus M_{i+2}, \, i\in \{2,3\};$\\
	$(10)~{\Omega_{i}}^{(3)}=M_i\oplus M_{i+1}, \, i\in \{6,7,9,11\}$.
	\end{thm A}

	Finally, based on the principle of the lifting method, we classify the finite-dimensional Hopf algebras over $H$ such that their infinitesimal braidings are those Yetter-Drinfeld modules listed $(1),(7),(8),(9),(10)$ in  \textbf{Theorem A}. The following is the second main result.
	\begin{thm B}
		Let $A$ be a finite-dimensional Hopf algebras over $H$ such that its infinitesimal braiding is isomorphic to one of the Yetter-Drinfeld modules $(1)$,~$(7)$,~$(8)$,~$(9)$,~$(10)$ in  \textbf{Theorem A}, then $A$ is isomorphic to one of the following\\
		$(1)$ $\mathfrak{U}_{1,1}(\lambda,\mu)$, see Definition $\ref{def1}$;
	
	$\mathfrak{U}_{1,2}(\lambda)$, see Definition $\ref{def2}$;
	
	$\mathfrak{U}_{1,3}(\lambda)$, see Definition $\ref{def3}$;
	
	$\mathfrak{U}_{1,4}(\lambda,\mu)$, see Definition $\ref{def4}$;
	
	$\mathfrak{U}_{1,5}(\lambda,\mu)$, see Definition $\ref{def5}$;
	
	$\mathfrak{U}_{1,6}(\lambda,\mu)$, see Definition $\ref{def6}$;
	
	$\mathfrak{U}_{1,7}(\lambda, \mu)$, see Definition $\ref{def7}$;
	
	$\mathfrak{U}_{1,8}(\lambda, \mu)$, see Definition $\ref{def8}$.
	
	\noindent
	$(2)$ $\mathfrak{U}_{1}(\lambda, \mu)$, see Definition $\ref{def9}$;
	
	$\mathfrak{U}_{2}(J)$, see Definition $\ref{def10}$;
	
	$\mathfrak{U}_{3}(I)$, see Definition
	$\ref{def11}$;

	$\mathfrak{U}_{6}(I)$, see Definition $\ref{def12}$;
	
	$\mathfrak{U}_{9}(I)$, see Definition $\ref{def13}$;
	
	$\mathfrak{U}_{10}(I)$, see Definition $\ref{def14}$;
	
	$\mathfrak{U}_{13}(I)$, see Definition $\ref{def15}$;
	
	$\mathfrak{U}_{14}(\lambda,\mu)$, see Definition $\ref{def16}$;
	
	$\mathfrak{U}_{15}(I)$, see Definition $\ref{def17}$;
	
	$\mathfrak{U}_{16}(\lambda,\mu)$, see Definition $\ref{def18}$;
	
	$\mathfrak{U}_{17}(I, \beta)$, see Definition $\ref{def19}$;
	
	$\mathfrak{U}_{19}(I)$, see Definition $\ref{def20}$.
	\end{thm B}
	\begin{rem}
		In \cite{AGM15}, authors constructed some examples of Hopf algebras with the dual Chevalley property by determining all semisimple Hopf algebras that are Morita-equivalent
		to finite group algebras. But in our case: the semisimple Hopf algebra $H_{b:x^2y}$ is a $2$-pseudo-cocycle twist of some group algebra rather than a $2$-cocycle twist, so it is no longer Morita-equivalent to a group algebra as semisimple algebras.
	\end{rem}
	
	\noindent\textbf{1.4.}
	The paper is organized as follows. In Section 2, we recall some
	basics and notations of Yetter-Drinfeld modules, Nichols algebras,
	Radford biproduct and Drinfeld double. In Section $3$, we recall
	the defition of $H$ and present the Drinfeld double $D=D(H^{cop})$
	by generators and relations. We also determine the simple
	$D$-modules. In Section 4, we describe the simple objects of
	$_H^H\mathcal{YD}$ by using the monoidal isomorphism of braided
	categories $_H^H\mathcal{YD}\cong {_D\mathcal{M}}$. In Section 5, we
	obtain all the possible finite-dimensional Nichols algebras of
	semisimple modules satisfying $\mathcal{B}(N)\cong \bigotimes_{i\in
		I}\mathcal{B}(N_i)$, for $N=\oplus_{i\in I}N_i$. In Section 6, based on the principle of the
	lifting method, we classify the finite-dimensional Hopf algebras
	over $H$ such that their infinitesimal braidings are those
	Yetter-Drinfeld modules listed in $(1)$, $(7)$, $(8)$, $(9)$, $(10)$ \textbf{Theorem A}. Then we get \textbf{Theorem B}.
	
		\section{Preliminaries}
	\noindent$\mathbf{Conventions}$.
	Throughout the paper, the ground field $\mathbb{K}$ is an algebraically closed field of characteristic zero and we denote by $\xi$ a primitive $4$-th root of unity. For the references of Hopf algebra theory, one can consult \cite{M}, \cite{CK}, \cite{R}, \cite{SW}, etc.
	
	If $H$ is a Hopf algebra over $\mathbb{K}$, then $\triangle,~\varepsilon, ~S$ denote the comultiplication, the counit and the antipode, respectively. We use Sweedler's notation for the comultiplication and coaction, $e.g.,$
	$\triangle(h)=h_{(1)}\otimes h_{(2)}$ for $h\in H$. We denote by $H^{op}$ the Hopf algebra with the opposite multiplication, by $H^{cop}$ the Hopf algebra with the opposite comultiplication, and by $H^{bop}$ the Hopf algebra $H^{op~cop}$. Denote by $G(H)$ the set of group-like elements of $H$.
	\begin{equation*}
		\mathcal{P}_{g,h}(H)=\{x\in H\mid \triangle(x)=x\otimes g+h\otimes x\},\hspace{1em}\forall ~g,~h\in G(H).
	\end{equation*}
	In particular, the linear space $\mathcal{P}(H)=\mathcal{P}_{1,1}(H)$ is called the set of primitive elements.
	
	If $V$ is a $\mathbb{K}$-vector space, $v\in V,~ f\in V^\ast$, we use either $f(v),~\langle f,v\rangle$ or $\langle v,f\rangle$ to denote the evaluation. Given $n\geq 0,$ we denote $\mathbb{Z}_n=\mathbb{Z}/n\mathbb{Z}$ and $\mathbb{I}_{0,n}=\{0,1,\cdots,n\}$. It is emphasized that the operations $ij$ and $i\pm j$ are considered modulo $n+1$ for $i,j\in \mathbb{I}_{0,n}$ when not specified.
	\subsection{Yetter-Drinfeld modules and Nichols algebras}
	Let $H$  be a Hopf algebra with bijective antipode. A left Yetter-Drinfeld module $V$ over $H$ is a left $H$-module $(V, \cdot)$ and a left $H$-module $(V, \delta)$ with $\delta(v)=v_{(-1)}\otimes v_{(0)}\in H\otimes V$ for all $v\in V$, satisfying
	\begin{center}
		$\delta(h\cdot v)=h_{(1)}v_{(-1)}S(h_{(3)})\otimes h_{(2)}\cdot v_{(0)}$,  \ $\forall \, v\in V, \ \forall \,h\in H.$
	\end{center}
	We denote by $^{H}_{H}\mathcal{YD}$ the category of finite-dimensional left Yetter-Drinfeld modules over $H$. It is a braided monoidal category: for $V,W\in ^{H}_{H}\mathcal{YD}$, the braiding $c_{V,W}:V\otimes W\to W\otimes V$ is given by
	\begin{flalign}\label{braiding}
		c_{V,W}(v\otimes w)=v_{(-1)}\cdot w\otimes v_{(0)}, \quad \forall \, v\in V, \ w\in W.
	\end{flalign}
	In particular, $(V,c_{V,V})$ is a braided vector space, that is, $c := c_{V,V}$ is a linear isomorphism satisfying the braid equation
	\begin{center}
		$(c\otimes \mathrm{id})(\mathrm{id} \otimes c)(c\otimes \mathrm{id})=(\mathrm{id} \otimes c)(c\otimes \mathrm{id})(\mathrm{id} \otimes c)$.
	\end{center}
	\begin{defi} $($\cite{AS2}, Definition 2.1$)$
		Let $H$ be a Hopf algebra and $V$ a Yetter-Drinfeld module over $H$. A braided $\mathbb{N}$-graded Hopf algebra $R=\bigoplus_{n\geq 0}R(n)$ in $^H_H\mathcal{YD}$ is called a Nichols algebra of $V$ if
		$\mathbb{K}\simeq R(0)$,  $V\simeq R(1)$, $R(1)=\mathcal{P}(R)$,  $R$ is generated by $R(1)$ as algebra.
	\end{defi}
	
	For any $V\in ^H_{H}\mathcal{YD}$ there is  a Nichols algebra $\mathcal{B}(V)$ associated to it. It is the quotient of the tensor algebra $T(V)$ by the maximal homogeneous two-sided ideal $I$ satisfying:
	$I$ is generated by homogeneous elements of degree greater than or equal to  two,
	$\Delta(I) \subseteq I\otimes T(V)+T(V)\otimes I$, i.e., it is also a coideal.
	In such case, $\mathcal{B}(V)=T(V)/I$  (\cite{AS2}, Section 2.1).
	\begin{rem}
		It is well konwn that the Nichols algebra $\mathcal{B}(V)$ is completely determined as an algebra and a coalgebra by the braided space $V$. If $W\subseteq V$ is a subspace such that $c(W\otimes W)\subseteq W\otimes W$, then $\dim \mathcal{B}(V)=\infty$ if $\dim\mathcal{B}(W)=\infty$. In particular, if $V$ contains  a nonzero element $v$ such that $c(v\otimes v)=v\otimes v$, then $\dim\mathcal{B}(V)=\infty$.
	\end{rem}\label{bra}
	\begin{lem}\label{nicdim} $($\cite{Gn00}, Theorem \,2.2 \,$)$
		Let $(V,c)$ be a braided vector space, $V=\bigoplus_{i=1}^{n}V_{i},$  such that each $\mathcal{B}(V_i)$ is finite-dimensional. Then $\dim \mathcal{B}(V)\geq \prod_{i=1}^{n}\dim \mathcal {B}(V_{i})$. Furthermore,  the equality holds if and only if $b_{ij}=b_{ji}^{-1}$ for all $i\neq j$, where $b_{ij}=c\mid_{V_{i}\otimes V_{j}}$.
	\end{lem}
	\subsection{Bosonization and Hopf algebras with a projection}
	Let $R$ be a Hopf algebra in $^H_H\mathcal{YD}$ and denote the coproduct by $\Delta_{R}(r)=r^{(1)}\otimes r^{(2)}$ for $r\in R$. We define the Radford biproduct or bosonization $R\sharp H$ as follows: as a vectoe space, $R\sharp H=R\otimes H$, and the multiplication and comultiplication are given by the smash product and smash-coprodcut, respectively:
	\[
	(r\sharp g)(s\sharp h)=r(g_{(1)}\cdot s) \sharp g_{(2)}h,\Delta(r\sharp g)=r^{(1)}\sharp (r^{(2)})_{(-1)}g_{(1)}\otimes (r^{(2)})_{(0)}\sharp g_{(2)}.
	\]		
	Clearly, the map $\iota : H \to R\sharp H,h\mapsto 1\sharp h,\  \forall \, h\in H,$ is injective and the map
	\[
	\pi : R\sharp H \to H,~~ r\sharp h\mapsto \epsilon_{R}(r)h,\quad \forall \, r\in R, \ b\in H.\]
	is surjective such that $\pi \circ \iota= \mathrm{id_{H}}.$ Moreover, it holds that
	\[
	R\cong R\sharp 1=(R\sharp H)^{co H}=\{ x\in R\sharp H\mid (\mathrm{id} \otimes \pi)\Delta(x)=x\otimes 1\}.\]
	
	Conversely, if $A$ is a Hopf algebra wih bijective antipode and $\pi :A\to H$ is a Hopf algebra epimorphism admittig  a Hopf algebra section $\iota :H\to A$ such that $\pi \circ \iota =\mathrm{id_{H}},$ then $R=A^{co \pi}$ is a braided Hopf algebra in $^H_H\mathcal {YD}$ and $A\simeq R\sharp H$ as Hopf algebra (see \cite{R85}, \cite{R}).
	\subsection{The Drinfeld double}
	Let $H$ be a finite-dimensional Hopf algebra over $\mathbb{K}$. The Drinfeld double $D(H)=H^{*cop}\otimes H$ is a Hopf algebra with the tensor product coalgebra structure and algebra structure defined by
	$$(p\otimes a)(q\otimes b)=p\langle q_{(3)},a_{(1)}\rangle q_{(2)}\otimes a_{(2)}\langle q_{(1),S^{-1}}(a_{(3)})\rangle b,\quad  \forall \, p,\  q\in H^{*}, \ a, \ b\in H.$$
	\begin{pro}$($\cite{M}, Proposition 10.6.16 $)$
		Let $H$ be a finite-dimensional Hopf algebra. Then the Yetter-Drinfeld category $^H_H\mathcal{YD}$ can be identified with the category $_{D(H^{cop})}\mathcal{M}$ of left modules over the Drinfeld double $D(H^{cop}).$
	\end{pro}
	\section{The Hopf algebra $H$ and its Drinfeld double}
	In this section, we will recall the structure of the 16-dimensional non-trivial semisimple Hopf algebra $H=H_{b:x^2y}$ \cite{K00} and present the Drinfeld Double $D=D(H^{cop})$ by generators and relations. Then we will determine all simple left $D$-modules.
	
	\begin{defi}\label{H}
		As an algebra $H$ is generated by $x,\, y, \,t$ satisfying the relations
		\begin{gather}
			x^4=1,\hspace{1em}y^2=1,\hspace{1em}t^2=x^2y,\hspace{1em}xy=yx,\hspace{1em}tx=x^{-1}t, \hspace{1em}ty=yt,\label{3.1}
		\end{gather}
		and its coalgebra structure is determinded by
		\begin{equation} \label{3.2}
			\begin{gathered}
				\triangle(x)=x\otimes x,\quad \triangle(y)=y\otimes y,\quad
				\triangle(t)=\frac{1}{2}(\,(1+y)t\otimes t+(1-y)t\otimes x^2 t\,),\\ \varepsilon(x)=\varepsilon(y)=1,\quad \varepsilon(t)=1,
			\end{gathered}
		\end{equation}
		and its antipode is determinded by
		\begin{gather}
			S(x)=x^{-1},\quad S(y)=y,\quad S(t)=\frac{1}{2}((1+y)x^2t-(1-y) t).\nonumber
		\end{gather}
	\end{defi}
	\begin{rem}
		$1.~ G(H)=\langle x\rangle \times \langle y\rangle,~\mathcal{P}_{1,g}(H)=\mathbb{K}\{1-g\}$ for $g\in G(H)$ and a linear basis of $H$ is given by $\{x^i, ~x^iy, ~x^it, ~x^iyt \mid i\in \mathbb{I}_{0,3}\}$.
		
		$2.~$ Denote by $\{(x^i)^*, ~(x^iy)^*, ~(x^it)^*, ~(x^iyt)^* \mid i\in \mathbb{I}_{0,3}\}$ the  dual basis of Hopf algebra $H^*$. Let
		\begin{flalign*}
			&a=(\,(1+x+x^2+x^3)(1-y)(1+\xi t)\,)^*,~~~b=(\,(1-x+x^2-x^3)(1+y)(1-t)\,)^*,\\
			&c=(\,(1+\xi x-x^2-\xi x^3)(1-y)(1+t)\,)^*.
		\end{flalign*}
		Then using the multiplication table induced by the relations of $H$, it follows that
		\begin{gather*}
			a^4=b^2=1,\quad c^2=b,\quad ab=ba, \quad ca=a^3c,\quad cb=bc,\\
			\triangle(a)=a\otimes a,\quad\triangle(b)=b\otimes b,\\
			\triangle(c)=\frac{1}{2}(c+a^2c)\otimes c+\frac{1}{2}(c-a^2c)\otimes a^2bc.
		\end{gather*}
	\end{rem}
	
	\begin{pro}
		The automorphic group of $H$ is given by
		\begin{equation*}
			\langle\tau_2\rangle\times\langle\tau_5\rangle\times \langle\tau_{13}\rangle\times\langle \tau_{17}\rangle\times \langle\tau_{33}\rangle\cong\mathbb{Z}_{4}\times\mathbb{Z}_{2}\times\mathbb{Z}_{4}\times \mathbb{Z}_2\times \mathbb{Z}_2.
		\end{equation*}
		Thus all automorphisms of $H$ are given in Table 1.
	\end{pro}
	\begin{center}
		\renewcommand{\arraystretch}{0.7}
		\begin{tabular}{|c|c|c|c|c|c|c|c|}
			\hline
			& $x$ & $y$ & $t$ &  & $x$ & $y$ & $t$   \\
			\hline
			$\tau_1$ & $x$ & $y$ & $t$ & $\tau_{33}$ & $xy$ & $y$ & $t$\\
			\hline
			$\tau_2$ & $x$ & $y$ & $xt$ & $\tau_{34}$ & $xy$ & $y$ & $xt$\\
			\hline
			$\tau_3$ & $x$ & $y$ & $x^2t$ & $\tau_{35}$ & $xy$ & $y$ & $x^2t$  \\
			\hline
			$\tau_4$ & $x$ & $y$ & $x^3t$ & $\tau_{36}$ & $xy$ & $y$ & $x^3t$  \\
			\hline
			$\tau_5$ & $x$ & $y$ & $yt$ & $\tau_{37}$ & $xy$ & $y$ & $yt$  \\
			\hline
			$\tau_6$ & $x$ & $y$ & $xyt$ & $\tau_{38}$ & $xy$ & $y$ & $xyt$    \\
			\hline
			$\tau_7$ & $x$ & $y$ & $x^2yt$ & $\tau_{39}$ & $xy$ & $y$ & $x^2yt$   \\
			\hline
			$\tau_8$ & $x$ & $y$ & $x^3yt$ & $\tau_{40}$ & $xy$ & $y$ & $x^3yt$  \\
			\hline
			$\tau_9$ & $x$ & $x^2y$ & $\frac{1}{2}((1+\xi)yt+(1-\xi)x^2yt)$ & $\tau_{41}$ & $xy$ & $x^2y$ & $\frac{1}{2}((1+\xi)yt+(1-\xi)x^2yt)$ \\
			\hline
			$\tau_{10}$ & $x$ & $x^2y$ & $\frac{1}{2}((1-\xi)yt+(1+\xi)x^2yt)$ & $\tau_{42}$ & $xy$ & $x^2y$ & $\frac{1}{2}((1-\xi)yt+(1+\xi)x^2yt)$  \\
			\hline
			$\tau_{11}$ & $x$ & $x^2y$ & $\frac{1}{2}((1+\xi)xyt+(1-\xi)x^3yt)$ & $\tau_{43}$ & $xy$ & $x^2y$ & $\frac{1}{2}((1+\xi)xyt+(1-\xi)x^3yt)$  \\
			\hline
			$\tau_{12}$ & $x$ & $x^2y$ & $\frac{1}{2}((1-\xi)xyt+(1+\xi)x^3yt)$ & $\tau_{44}$ & $xy$ & $x^2y$ & $\frac{1}{2}((1-\xi)xyt+(1+\xi)x^3yt)$ \\
			\hline
			$\tau_{13}$ & $x$ & $x^2y$ &  $\frac{1}{2}((1+\xi)t+(1-\xi)x^2t)$& $\tau_{45}$ & $xy$ & $x^2y$ & $\frac{1}{2}((1+\xi)t+(1-\xi)x^2t)$\\
			\hline
			$\tau_{14}$ & $x$ & $x^2y$ & $\frac{1}{2}((1-\xi)t+(1+\xi)x^2t)$ & $\tau_{46}$ & $xy$ & $x^2y$ & $\frac{1}{2}((1-\xi)t+(1+\xi)x^2t)$  \\
			\hline
			$\tau_{15}$ & $x$ & $x^2y$ & $\frac{1}{2}((1+\xi)xt+(1-\xi)x^3t)$ & $\tau_{47}$ & $xy$ & $x^2y$ & $\frac{1}{2}((1+\xi)xt+(1-\xi)x^3t)$ \\
			\hline
			$\tau_{16}$ & $x$ & $x^2y$ & $\frac{1}{2}((1-\xi)xt+(1+\xi)x^3t)$ & $\tau_{48}$ & $xy$ & $x^2y$ & $\frac{1}{2}((1-\xi)xt+(1+\xi)x^3t)$  \\
			\hline
			$\tau_{17}$ & $x^3$ & $y$ & $t$ & $\tau_{49}$ & $x^3y$ & $y$ & $t$\\
			\hline
			$\tau_{18}$ & $x^3$ & $y$ & $xt$ & $\tau_{50}$ & $x^3y$ & $y$ & $xt$\\
			\hline
			$\tau_{19}$ & $x^3$ & $y$ & $x^2t$ & $\tau_{51}$ & $x^3y$ & $y$ & $x^2t$  \\
			\hline
			$\tau_{20}$ & $x^3$ & $y$ & $x^3t$ & $\tau_{52}$ & $x^3y$ & $y$ & $x^3t$  \\
			\hline
			$\tau_{21}$ & $x^3$ & $y$ & $yt$ & $\tau_{53}$ & $x^3y$ & $y$ & $yt$  \\
			\hline
			$\tau_{22}$ & $x^3$ & $y$ & $xyt$ & $\tau_{54}$ & $x^3y$ & $y$ & $xyt$    \\
			\hline
			$\tau_{23}$ & $x^3$ & $y$ & $x^2yt$ & $\tau_{55}$ & $x^3y$ & $y$ & $x^2yt$   \\
			\hline
			$\tau_{24}$ & $x^3$ & $y$ & $x^3yt$ & $\tau_{56}$ & $x^3y$ & $y$ & $x^3yt$  \\
			\hline
			$\tau_{25}$ & $x^3$ & $x^2y$ & $\frac{1}{2}((1+\xi)yt+(1-\xi)x^2yt)$ & $\tau_{57}$ & $x^3y$ & $x^2y$ & $\frac{1}{2}((1+\xi)yt+(1-\xi)x^2yt)$ \\
			\hline
			$\tau_{26}$ & $x^3$ & $x^2y$ & $\frac{1}{2}((1-\xi)yt+(1+\xi)x^2yt)$ & $\tau_{58}$ & $x^3y$ & $x^2y$ & $\frac{1}{2}((1-\xi)yt+(1+\xi)x^2yt)$  \\
			\hline
			$\tau_{27}$ & $x^3$ & $x^2y$ & $\frac{1}{2}((1+\xi)xyt+(1-\xi)x^3yt)$ & $\tau_{59}$ & $x^3y$ & $x^2y$ & $\frac{1}{2}((1+\xi)xyt+(1-\xi)x^3yt)$  \\
			\hline
			$\tau_{28}$ & $x^3$ & $x^2y$ & $\frac{1}{2}((1-\xi)xyt+(1+\xi)x^3yt)$ & $\tau_{60}$ & $x^3y$ & $x^2y$ & $\frac{1}{2}((1-\xi)xyt+(1+\xi)x^3yt)$ \\
			\hline
			$\tau_{29}$ & $x^3$ & $x^2y$ &  $\frac{1}{2}((1+\xi)t+(1-\xi)x^2t)$& $\tau_{61}$ & $x^3y$ & $x^2y$ & $\frac{1}{2}((1+\xi)t+(1-\xi)x^2t)$\\
			\hline
			$\tau_{30}$ & $x^3$ & $x^2y$ & $\frac{1}{2}((1-\xi)t+(1+\xi)x^2t)$ & $\tau_{62}$ & $x^3y$ & $x^2y$ & $\frac{1}{2}((1-\xi)t+(1+\xi)x^2t)$  \\
			\hline
			$\tau_{31}$ & $x^3$ & $x^2y$ & $\frac{1}{2}((1+\xi)xt+(1-\xi)x^3t)$ & $\tau_{63}$ & $x^3y$ & $x^2y$ & $\frac{1}{2}((1+\xi)xt+(1-\xi)x^3t)$ \\
			\hline
			$\tau_{32}$ & $x^3$ & $x^2y$ & $\frac{1}{2}((1-\xi)xt+(1+\xi)x^3t)$ & $\tau_{64}$ & $x^3y$ & $x^2y$ & $\frac{1}{2}((1-\xi)xt+(1+\xi)x^3t)$  \\
			\hline
		\end{tabular}
	\end{center}		
	\begin{proof}
		Let $f$ be an automorphism of $H$. Since $G(H)=\langle x\rangle \times \langle y\rangle$, and $x^4=1$, $y^2=1$, we have that $f(x)\in \{ x, x^3, xy, x^3y\}$, $f(y)\in \{y, x^2y\}$. If $f(t)$ has group-like elements, then it does not satisfy $\Delta f(t)=(f\otimes f)\Delta (t)$. Then we let $f(t)=(k_1+k_2x+k_3x^2+k_4x^3+k_5y+k_6xy+k_7x^2y+k_8x^3y)t$, $k_i\in \mathbb{K}$, $1\leq i \leq 8$. Suppose $f(x)=x,f(y)=y$, using  $\Delta f(t)=(f\otimes f)\Delta(t)$ and comparing cofficients of $x^iy^jt\otimes x^{i'}y^{j'}t$, $i,i^{'}\in \mathbb {I}_{0,3},j,j^{'}\in \mathbb{I}_{0,1}$,  we obtain the following:
		\begin{gather*}
			k_1(k_3-k_7)=k_1(k_3+k_7)=k_1(k_4+k_8)=k_1(k_2+k_6)=0,\\
			k_2(k_4-k_8)=k_2(k_4+k_8)=k_2(k_1+k_5)=k_2(k_3+k_7)=0,\\
			k_3(k_1-k_5)=k_3(k_1+k_5)=k_3(k_2+k_6)=k_3(k_4+k_8)=0,\\
			k_4(k_2-k_6)=k_4(k_2+k_6)=k_4(k_1+k_5)=k_4(k_3+k_7)=0,\\
			k_5(k_3-k_7)=k_5(k_3+k_7)=k_5(k_2+k_6)=k_5(k_4+k_8)=0,\\
			k_6(k_4-k_8)=k_6(k_4+k_8)=k_6(k_1+k_5)=k_6(k_3+k_7)=0,\\
			k_7(k_1-k_5)=k_7(k_1+k_5)=k_7(k_2+k_6)=k_7(k_4+k_8)=0,\\
			k_8(k_2-k_6)=k_8(k_2+k_6)=k_8(k_1+k_5)=k_8(k_3+k_7)=0.
		\end{gather*}
		Therefore,  we have that  $f(t)\in \{t,xt,x^2t,x^3t,yt,xyt,x^2yt,x^3yt\}$.\\
		When $f(x)=x,f(y)=x^2y$, by $\Delta f(t)=(f\otimes f)\Delta (t)$ and comparing cofficients of $x^iy^jt\otimes x^{i'}y^{j'}t$, $i,i^{'}\in \mathbb {I}_{0,3},j,j^{'}\in \mathbb{I}_{0,1}$, then
		\begin{gather*}
			k_1(k_2+k_8)=k_1(k_4+k_6)=k_1(k_1+k_3+k_5+k_7-1)=0,\\
			k_2(k_1+k_7)=k_2(k_3+k_5)=k_2(k_2+k_4+k_6+k_8-1)=0,\\
			k_3(k_2+k_8)=k_3(k_4+k_6)=k_3(k_1+k_3+k_5+k_7-1)=0,\\
			k_4(k_3+k_5)=k_4(k_1+k_7)=k_4(k_2+k_4+k_6+k_8-1)=0,\\
			k_5(k_2+k_8)=k_5(k_4+k_6)=k_5(k_1+k_3+k_5+k_7-1)=0,\\
			k_6(k_1+k_7)=k_6(k_3+k_5)=k_6(k_2+k_4+k_6+k_8-1)=0,\\
			k_7(k_2+k_8)=k_7(k_4+k_6)=k_7(k_1+k_3+k_5+k_7-1)=0,\\
			k_8(k_1+k_7)=k_8(k_3+k_5)=k_8(k_2+k_4+k_6+k_8-1)=0.
		\end{gather*}
		So we have
		\allowdisplaybreaks
		\begin{align*}
			f(t) \in \Bigl\{
			&\frac12(\,(1+\xi)t + (1-\xi)x^2 t\,),\,
			\frac12(\,(1-\xi)t + (1+\xi)x^2 t\,),\\
			&\frac12(\,(1+\xi)xt + (1-\xi)x^3 t\,),\,
			\frac12(\,(1-\xi)xt + (1+\xi)x^3 t\,),\\
			&\frac12(\,(1+\xi)yt + (1-\xi)x^2 yt\,),\,
			\frac12(\,(1-\xi)yt + (1+\xi)x^2 yt\,),\\
			&\frac12(\,(1+\xi)xyt + (1-\xi)x^3 yt\,),\,
			\frac12(\,(1-\xi)xyt + (1+\xi)x^3 yt\,)
			\Bigr\}.
		\end{align*}

		The remaining cases for  $f(x)$ and $f(y)$  can be treated analogously. Hence, This completes our proof.
	\end{proof}

	Now we describe the Drinfeld Double $D(H^{cop})$ of $H^{cop}$.
	\begin{pro}
		$D:=D(H^{cop})$ as a coalgebra is isomorphic to tensor coalgebra $H^{*bop}\otimes H^{cop}$, and as an algebra is generated by the elements $a,~b,~c,~x,~y,~t$ such that $x,~y,~t$ satisfying the relations of $H^{cop}$, $a,~b,~c$ satisfying the relations of $H^{*bop}$ and
		\begin{gather*}
			xa=ax, \quad xb=bx, \quad xc=a^2cx,\\
			ya=ay,\quad yb=by, \quad yc=cy,\\
			ta=ax^2t,\quad tb=bt, \quad tc=a^2bcx^2yt.
		\end{gather*}
	\end{pro}
	\begin{proof}
		After a direct computation, we have that
		\begin{flalign*}
			\triangle_{H^{cop}}^2(x)=&x\otimes x\otimes x, ~~~\triangle_{H^{cop}}^2(y)=y\otimes y\otimes y,\\
			\triangle_{H^{*bop}}^2(a)=&a\otimes a\otimes a,
			~~~\triangle_{H^{*bop}}^2(b)=b\otimes b\otimes b,\\
			\triangle_{H^{cop}}^2(t)=&\frac{1}{4}((t+x^2t)\otimes t\otimes (t+yt)+(t+x^2t)\otimes x^2t\otimes(t-yt)\\
			&+(t-x^2t)\otimes yt\otimes
			(t+yt)+(t-x^2t)\otimes x^2yt\otimes (yt-t)),\\
			\triangle_{H^{*bop}}^2(c)=&\frac{1}{4}((c+a^2c)\otimes c\otimes (c+a^2bc)+(c-a^2c)\otimes a^2bc \otimes (c+a^2bc)\\
			&+(c+a^2c)\otimes a^2c
			\otimes (c-a^2bc)+(a^2c-c)\otimes bc\otimes (c-a^2bc)).
		\end{flalign*}
		It follows that
		\begin{flalign*}
			xa&=\langle a,x\rangle ax\langle a,S(x)\rangle =ax,~~~~xb=\langle b,x\rangle bx \langle b,S(x)\rangle ,\\
			xc&=\frac{1}{4}(\langle c+a^2c,x\rangle(c+a^2x)x\langle c,S(x)\rangle+\langle c+a^2c,x\rangle (c-a^2c)x\langle a^2bc, S(x)\rangle)=a^2cx,\\
			ya&=\langle a,y\rangle ay \langle a, S(y) \rangle =ay,~~~~yb=\langle b, y\rangle by \langle b, S(y)\rangle=by,\\
			yc&= \frac{1}{4}(\langle c+a^2c, y\rangle (c+a^2c)y \langle c, S(y)\rangle +\langle c+a^2c, y\rangle (c-a^2c)y \langle a^2bc, S(y)\rangle )=cy,\\
			ta&=\frac{1}{4}(\langle a, t+x^2t\rangle at \langle a, S(t+yt)\rangle+\langle a, t+x^2t\rangle  ax^2t \langle a,S(t-yt)\rangle )=ax^2t,\\
			tb&=\frac{1}{4}(\langle b, t\rangle b(t+yt)\langle b,S(t+yt)\rangle +\langle b, x^2t \rangle b(t-yt)\langle b, S(t+yt)\rangle )=bt,\\
			tc&=\frac{1}{16}(\langle c+a^2bc ,t{-}x^2t\rangle cyt \langle c+a^2c, S(t+yt)\rangle -\langle c+a^2bc,t {-}x^2t\rangle cx^2yt \langle c+a^2c,S(t{-}yt)\rangle \\
			&+\!\langle c+\!a^2bc, \!t{-}\!x^2t\rangle a^2bcyt\langle c-\!a^2c,\!S(t+\!yt)\rangle-\!\langle c+\!a^2bc,t{-}\!x^2t\rangle a^2bcx^2yt\langle c-\!a^2c,\!S(t{-}\!yt)\rangle)\\
			&=a^2bcx^2yt.
		\end{flalign*}
		This completes the proof.
		
	\end{proof}

	We begin by describing the one-dimensional $D$-modules.
	\begin{lem}\label{1simple}
		There are $32$  pairwise non-isomorphic one-dimensional simple modules $\mathbb{K}_{\chi_{i,j,k,l}}$ given by the characters $\chi_{i,j,k,l}$
		\begin{gather*}
			\chi_{i,j,k,l}(x)=(-1)^i,\quad \chi_{i,j,k,l}(y)=(-1)^j, \quad \chi_{i,j,k,l}(t)=\xi ^j,\\
			\chi_{i,j,k,l}(a)=(-1)^k, \quad \chi_{i,j,k,l}(b)=(-1)^j,\quad \chi_{i,j,k,l}(c)=(-1)^l\xi^j,
		\end{gather*}
		where $0\leq i,~k,~l\leq 1, ~0\leq j\leq 3$.
		
		Moreover, any one-dimensional $D$-module is isomorphic to $\mathbb{K}_{\chi_{i,j,k,l}}$, for all $0\leq i$, $k$, $l\leq 1$, $0\leq j\leq 3$.
	\end{lem}
	\begin{proof} Let $\lambda :D\to \mathbb{K}$ be a character and set
		$$\lambda(x)=\lambda_1,\quad \lambda(y)=\lambda_2,\quad \lambda(t)=\lambda_3, \quad \lambda(a)=\lambda_4, \quad \lambda(b)=\lambda_5,\quad \lambda(t)=\lambda_6.$$
		From  $x^4=y^2=a^4=b^2=1$, we have  $\lambda_1^4=\lambda_2^2=\lambda_4^4=\lambda_5^2=1$. Since $tx=x^{-1}t$,  we get  $\lambda_3\lambda_1=\lambda_1^3\lambda_3$, which implies $\lambda_1^2=1$. Moreover $t^2=x^2y$,  we have  $\lambda_3^2=\lambda_1^2\lambda_2$. Using $\lambda_1^2=1$, we have $\lambda_3^2=\lambda_2$.  This implies $\lambda_3^4=1$. It follows from $ac=ca^3$ that  $\lambda_4^2=1$. Finally, $tc=a^2bcx^2yt$ implies $\lambda_2\lambda_5=1$. Therefore, $\lambda $ is completely determined by $\lambda(x)$, $\lambda(t)$, $\lambda(a)$, $\lambda(b)$. Accordingly, we set
		$$\lambda(x)=(-1)^i,~~\lambda(t)=\xi ^j,~~\lambda(a)=(-1)^k,~~\lambda(b)=(-1)^j,$$
		where  $0\leq i,~k,~l\leq 1, ~0\leq j\leq 3.$
	It is straightforward to verify that these characters are pairwise non-isomorphic
	 and any one-dimensional $D$ module is isomorphic to $\mathbb{K}_{\chi_{i,j,k,l}},$ for all $0\leq i,~k,~l\leq 1, ~0\leq j\leq 3.$
	\end{proof}
	
	Next,  we describe the  two-dimensional simple $D$-modules. For convenience, we denote
	\begin{flalign*}&\Omega^1=\{(0,j,k,l)\mid j\in \mathbb{Z}_4, k\in \{1,3\},l\in\mathbb{Z}_2\},\\
		&\Omega^2=\{(i,j,k,l)\mid i\in \mathbb{Z}_2, j\in \mathbb{Z}_{2},k\in \{0,2\},j+l\equiv 1\, \mathrm{mod}\,2\},\\
		&\Omega=\Omega^1\cup\Omega^2,\\
		&\Lambda^1=\{(1,j,k)\mid j,k\in \mathbb{Z}_{2}\}\cup\{(3,j,k)\mid  j,k\in \mathbb{Z}_{2}, j+k\equiv 0 \,\mathrm{mod} \,2\},\\
		&\Lambda^2=\{(1,j,k)\mid j,k\in\mathbb{Z}_{2},j+k\equiv 1\, \mathrm{mod} \,2\},\\
		&\Gamma=\{(1,j,k,l)\mid j,k\in \mathbb{Z}_{2},l\in \mathbb{Z}_4\},\\
		&|\Omega|=24,|\Lambda^1|=6,|\Lambda^2|=2,|\Gamma|=16.	\end{flalign*}

	\begin{lem}\label{2simple}
		For any $4$-tuple $(i,j,k,l)\in \Omega$, there exists a simple two-dimensional  $D$-module $V_{i,j,k,l}$ such that, with respect to a fixed basis, the action of the
		generators of $D$ is given by
		\begin{equation*}
			\begin{split}
			&[x]=
			\left(
			\begin{array}{cc}
				(-1)^i & 0 \\
				0 & (-1)^{i+k} \\
			\end{array}
			\right),
			\quad [y]=
			\left(
			\begin{array}{cc}
				(-1)^j & 0 \\
				0 & (-1)^{j} \\
			\end{array}
			\right),\quad
			[t]=
			\left(
			\begin{array}{cc}
				\xi^j & 0 \\
				0 & (-1)^{j+k+l}\xi^j \\
			\end{array}
			\right),\\
			&[a]=
			\left(
			\begin{array}{cc}
				\xi^k & 0 \\
				0 & (-\xi)^k \\
			\end{array}
			\right),
			\quad [b]=
			\left(
			\begin{array}{cc}
				(-1)^l & 0 \\
				0 & (-1)^l \\
			\end{array}
			\right),
			\quad[c]=
			\left(
			\begin{array}{cc}
				0 & 1 \\
				(-1)^l & 0 \\
			\end{array}
			\right).
			\end{split}
		\end{equation*}
		For any $3$-tuple $(i,j,k)\in \Lambda^1$, there exists a simple two-dimensional  $D$-module  $W^1_{i,j,k}$  such that, with respect to a fixed basis, the action of the
		generators of $D$ is given by
		\begin{flalign*}
			&[x]=
			\left(
			\begin{array}{cc}
				\xi^i & 0 \\
				0 & (-\xi)^i\\
			\end{array}
			\right),
			~~~~~[y]=
			\left(
			\begin{array}{cc}
				(-1)^j & 0 \\
				0 & (-1)^{j} \\
			\end{array}
			\right),
			~~~~~[t]=
			\left(
			\begin{array}{cc}
				0 & \xi^i \\
				(-1)^j\xi^i & 0 \\
			\end{array}
			\right),\\
			&[a]=
			\left(
			\begin{array}{cc}
				\xi^i & 0 \\
				0 & (-\xi)^{i} \\
			\end{array}
			\right),
			~~~~~~[b]=
			\left(
			\begin{array}{cc}
				(-1)^k & 0 \\
				0 & (-1)^k \\
			\end{array}
			\right),
			~~~~~[c]=
			\left(
			\begin{array}{cc}
				0 &1 \\
				(-1)^k & 0 \\
			\end{array}
			\right).
		\end{flalign*}
		For any $3$-tuple $(i,j,k)\in \Lambda^1$, there exists a simple two-dimensional  $D$-module $W^2_{i,j,k}$  such that, with respect to a fixed basis, the action of the
		generators of $D$ is given by
		\begin{flalign*}
			&[x]=
			\left(
			\begin{array}{cc}
				\xi^i & 0 \\
				0 & (-\xi)^i\\
			\end{array}
			\right),
			~~~~~[y]=
			\left(
			\begin{array}{cc}
				(-1)^j & 0 \\
				0 & (-1)^{j} \\
			\end{array}
			\right),
			~~~~~[t]=
			\left(
			\begin{array}{cc}
				0 & \xi^i \\
				(-1)^j\xi^i & 0 \\
			\end{array}
			\right),\\
			&[a]=
			\left(
			\begin{array}{cc}
				(-\xi)^{i} & 0 \\
				0 & \xi^{i} \\
			\end{array}
			\right),
			~~~~~~[b]=
			\left(
			\begin{array}{cc}
				(-1)^k & 0 \\
				0 & (-1)^k \\
			\end{array}
			\right),
			~~~~~[c]=
			\left(
			\begin{array}{cc}
				0 &1 \\
				(-1)^k & 0 \\
			\end{array}
			\right).
		\end{flalign*}
		For any $3$-tuple $(i,j,k)\in \Lambda^2$, there exists a simple two-dimensional  $D$-module $W^3_{i,j,k}$  such that, with respect to a fixed basis, the action of the
		generators of $D$ is given by
		\begin{flalign*}
			&[x]=
			\left(
			\begin{array}{cc}
				\xi^i & 0 \\
				0 & (-\xi)^i\\
			\end{array}
			\right),
			~~~~~[y]=
			\left(
			\begin{array}{cc}
				(-1)^j & 0 \\
				0 & (-1)^{j} \\
			\end{array}
			\right),
			~~~~~[t]=
			\left(
			\begin{array}{cc}
				0 & (-\xi)^i \\
				(-1)^{i+j}\xi^i & 0 \\
			\end{array}
			\right),\\
			&[a]=
			\left(
			\begin{array}{cc}
				\xi^{i} & 0 \\
				0 & (-\xi)^{i} \\
			\end{array}
			\right),
			~~~~~~[b]=
			\left(
			\begin{array}{cc}
				(-1)^k & 0 \\
				0 & (-1)^k \\
			\end{array}
			\right),
			~~~~~[c]=
			\left(
			\begin{array}{cc}
				0 &1 \\
				(-1)^k & 0 \\
			\end{array}
			\right).
		\end{flalign*}
		For any $3$-tuple $(i,j,k)\in \Lambda^2$, there exists a simple two-dimensional  $D$-module $W^4_{i,j,k}$ such that, with respect to a fixed basis, the action of the
		generators of $D$ is given by
		\begin{flalign*}
			&[x]=
			\left(
			\begin{array}{cc}
				\xi^i & 0 \\
				0 & (-\xi)^i\\
			\end{array}
			\right),
			~~~~~[y]=
			\left(
			\begin{array}{cc}
				(-1)^j & 0 \\
				0 & (-1)^{j} \\
			\end{array}
			\right),
			~~~~~[t]=
			\left(
			\begin{array}{cc}
				0 & (-\xi)^i \\
				(-1)^{i+j}\xi^i & 0 \\
			\end{array}
			\right),\\
			&[a]=
			\left(
			\begin{array}{cc}
				(-\xi)^{i} & 0 \\
				0 & \xi^{i} \\
			\end{array}
			\right),
			~~~~~~[b]=
			\left(
			\begin{array}{cc}
				(-1)^k & 0 \\
				0 & (-1)^k \\
			\end{array}
			\right),
			~~~~~[c]=
			\left(
			\begin{array}{cc}
				0 &1 \\
				(-1)^k & 0 \\
			\end{array}
			\right).
		\end{flalign*}
		For any $4$-tuple $(i,j,k,l)\in \Gamma$, there exists a simple two-dimensional  $D$-module $U_{i,j,k,l}$  such that, with respect to a fixed basis, the action of the
		generators of $D$ is given by
		\begin{equation*}
			\begin{split}
			&[x]=
			\left(
			\begin{array}{cc}
				\xi^i & 0 \\
				0 & (-\xi)^{i} \\
			\end{array}
			\right),
			\quad[y]=
			\left(
			\begin{array}{cc}
				(-1)^j & 0 \\
				0 & (-1)^{j} \\
			\end{array}
			\right),
			\quad [t]=
			\left(
			\begin{array}{cc}
				0 & 1 \\
				(-1)^{i+j} & 0 \\
			\end{array}
			\right),\\
			&[a]=
			\left(
			\begin{array}{cc}
				(-1)^k & 0 \\
				0 & (-1)^{i+k} \\
			\end{array}
			\right),
			\quad [b]=
			\left(
			\begin{array}{cc}
				(-1)^l & 0 \\
				0 & (-1)^l \\
			\end{array}
			\right),
			\quad[c]=
			\left(
			\begin{array}{cc}
				\xi^l & 0 \\
				0& (-1)^{i+j+l}\xi^l\\
			\end{array}
			\right).
			\end{split}
		\end{equation*}
		Conversely, every simple two-dimensional   $D$-module  $V$ is
		isomorphic to one of the following modules $V_{i,j,k,l}$ with $(i,j,k,l)\in \Omega$,  $W^p_{i,j,k}$ with $(i,j,k)\in \Lambda^q$ and $p\in \mathbb{I}_{1,4}$, or $U_{i,j,k,l}$ with  $(i, j, k, l)\in \Gamma$.  Finally,  the isomorphic classes of the above simple two-dimensional
		$D$-modules  are pairwise
		non-isomorphic.
		 In fact,
		$V_{i,j,k,l}\cong V_{i', j', k', l'}$ if and only if $(i,j,k,l)=(i', j', k', l')$ $\in$ $\Omega$; $W^p_{i,j,k}\cong W^p_{i', j', k'}$, $p\in \mathbb{I}_{1,4}$ if and only if $(i,j,k)=(i', j', k')$ $\in$ $\Lambda^q$; $U_{i,j,k,l}\cong U_{i', j', k', l'}$ if and only if $(i,j,k,l)=(i', j', k', l')\in \Gamma$.
	\end{lem}
	\begin{proof}  Let $V$ be a simple two-dimensional $D$-module.
		Since   the generators $x$, $y$, $a$, $b$ $\in D$ commute with each other and satisfy $x^4=y^2=a^4=b^2=1$,  the action of  the generators of $D$ on  $V$  is represented by matrices of the following form
		\begin{flalign*}
			&[x]=
			\left(
			\begin{array}{cc}
				x_1 & 0 \\
				0 & x_2 \\
			\end{array}
			\right),
			\quad[y]=
			\left(
			\begin{array}{cc}
				y_1 & 0 \\
				0 & y_2\\
			\end{array}
			\right),
			\quad[t]=
			\left(
			\begin{array}{cc}
				t_1 & t_2 \\
				t_3 & t_4 \\
			\end{array}
			\right),\\
			&[a]=
			\left(
			\begin{array}{cc}
				a_1 & 0 \\
				0 & a_2\\
			\end{array}
			\right),
			\quad[b]=
			\left(
			\begin{array}{cc}
				b_1 & 0 \\
				0 & b_2 \\
			\end{array}
			\right),
			\quad[c]=
			\left(
			\begin{array}{cc}
				c_1 & c_2 \\
				c_3& c_4\\
			\end{array}
			\right),
		\end{flalign*}
		where $x_1^4=x_2^4=y_1^2=y_2^2=a_1^4=a_2^4=b_1^2=b_2^2=1.$ The relations $ty=yt$ and $yc=cy$ imply
		 that
		\begin{center}$t_2(y_1-y_2)=0, \quad t_3(y_1-y_2)=0, \quad c_2(y_1-y_2)=0, \quad c_3(y_1-y_2)=0.$
		\end{center}
		Suppose  that  $y_1\neq y_2$. Then $t_2=t_3=c_2=c_3=0$. However, $V$ is  simple if and only if  $t_2\neq 0$, and $t_3\neq 0$, $c_2\neq 0$ and $c_3\neq 0$. This yields a contradiction. Thus $y_1=y_2$. Similarly, From the relations $bc=cb, ~bt=tb$, we have
		\begin{center}	$c_2(b_1-b_2), \quad c_3(b_1-b_2)=0, \quad t_2(b_1-b_2)=0, \quad t_3(b_1-b_2)=0$.
		\end{center}
		The same argument as above shows that $b_1=b_2$.
		By $t^2=x^2y$, $tx=x^{-1}t$, we have
		\begin{flalign}
			&t_1^2+t_2t_3=x_1^2y_1, \quad t_2t_3+t_4^2=x_2^2y_1, \quad t_2(t_1+t_4)=0, \quad t_3(t_1+t_4)=0,\label{r1}\\
			&x_1t_1(1-x_1^2)=0, \quad x_2t_2(1-x_2^2)=0, \quad t_2(x_2-x_1^3)=0, \quad x_1(t_3-x_2^3)=0\label{r2}.
		\end{flalign}
		Now, using $ta=ax^2t$ and $tc=a^2bcx^2yt$, we have
		\begin{flalign}
			&t_1a_1(1-x_1^2)=0, \quad t_4a_2(1-x_2^2)=0, \quad t_2(a_2-a_1x_1^2)=0, \quad t_3(a_1-a_2x_2^2)=0,\label{r3}\\
			&t_1c_1+t_2c_3=a_1^2b_1y_1(c_1x_1^2t_1+c_2x_2^2t_3), \quad t_1c_2+t_2c_4=a_1^2b_1y_1(c_1x_1^2t_2+c_2x_2^2t_4)\label{r4},\\
			&t_3c_1+t_4c_3=a_2^2b_1y_1(c_3x_1^2t_1+c_4x_2^2t_3), \quad t_3c_2+t_4c_4=a_2^2b_1y_1(c_3x_1^2t_2+c_4x_2^2t_4)\label{r5}.
		\end{flalign}
		The remaining relations $c^2=b$, $ca=a^3c$ and $xc=a^2cx$ yield
		\begin{flalign}
			&c_1^2+c_2c_3=b_1, \quad c_2(c_1+c_4)=0, \quad c_3(c_1+c_4)=0, \quad c_2c_3+c_4^2=b_1\label{r6},\\
			&a_1c_1(1-a_1^2)=0, \quad a_2c_4(1-a_2^2)=0, \quad c_2(a_2-a_1^3)=0, \quad c_3(a_1-a_2^3)=0\label{r7},\\
			&x_1c_1(1-a_1^2)=0, \quad x_2c_4(1-a_2^2)=0, \quad c_2(x_1-a_1^2x_2)=0, \quad c_3(x_2-a_2^2x_1)=0.\label{r8}
		\end{flalign}

		Suppose that $t_1t_4\neq 0$ and $t_2=t_3=0$. From  $(\ref{r1})-(\ref{r5})$, it follows   that
		\begin{flalign}
			&t_1^2=x_1^2y_1, \quad t_4^2=x_2^2y_1, \quad x_1^2=x_2^2=1, \quad (1-a_1^2b_1x_1^2y_1)c_1=0,\label{r9}\\
			&(1-a_2^2b_1x_2^2y_1)c_4=0, \quad c_2(t_1-a_1^2b_1x_2^2y_1t_4)=0, \quad c_3(t_4-a_2^2b_1x_1^2y_1t_1)=0.\label{r10}
		\end{flalign}
		Since $V$ is simple, $c_2c_3\neq 0$. Using   $(\ref{r6})-(\ref{r10})$,  we have  $$c_1+c_4=0,\, a_2=a_1^3, \, a_1=a_2^3, \, x_1=a_1^2x_2,\,  x_2=a_2^2x_1, \, t_1=a_1^2b_1x_2^2y_1t_4, \, t_4=a_2^2b_1x_1^2y_1t_1.$$
		In this case, suppose that $c_1=c_4=0$. Then $c_2c_3=b_1$ is immediate. Without loss of generality, we may assume that $c_2=1,c_3=b_1$. Therefore, the action of the generators of $D$ on $V$ is represented by matrices of the form $V_{i,j,k,l},$ where $(i,j,k,l)\in \Omega$. On the other hand, suppose that $c_1c_4\neq 0$. Then $V$ is  not simple and the proof is straightforward.
		
	We now consider the case $t_1 = t_4 = 0$.
	It follows immediately  that $t_2\neq 0,  t_3\neq 0$.  As in the previous case, by$(\ref{r1})-(\ref{r5})$, we have
		\begin{equation*}
			t_2t_3=x_1^2y_1, \quad x_2=x_1^3, \quad a_2=a_1x_1^2, \quad c_4=a_1^2b_1y_1c_1x_1^2, \quad t_3c_2=a_2^2b_1y_1c_3x_1^2t_2.
		\end{equation*}
		Suppose  further that
	 $c_1c_4=0$. This implies that
		\begin{equation*}
			a_2=a_1^3, \quad x_2=a_2^2x_1.
		\end{equation*}
		 Thus, the action of the generators of $D$ on $V$ is represented by matrices of the form $W^p_{i,j,k}$, where $(i,j,k)\in \Lambda^q$, $p\in\mathbb{I}_{1,4}$, $q\in \mathbb{I}_{1,2}$. On the other hand,
		suppose that $c_1c_4\neq 0$, $c_2c_3=0$, we obtain
		\begin{gather*}
			c_1^2=c_4^2=b_1,\quad c_4=a_1^2b_1y_1c_1x_1^2.
		\end{gather*}
	the action of the generators of $D$ on $V$ is represented by matrices of the form $U_{i,j,k,l}$, where $(i,j,k,l)\in \Gamma$.
	The remaining cases  lead to the same conclusions as those obtained above.
Moreover, for any $(i,j,k,l)\in\Omega$, the representation matrices $V_{i,j,k,l}$ are not equivalent to either of the two types of representation matrices obtained above. Furthermore,
$W_{i,j,k}\ncong U_{i',j',k', l'}$
for any $(i,j,k)\in\Lambda^p$, $p\in\mathbb{I}_{1,4}$ and $(i',j',k',l')\in\Gamma$.

		We claim that $V_{i,j,k,l}\cong V_{i^{'},j^{'},k^{'},l^{'}}$ if and only if $(i,j,k,l)=(i^{'},j^{'},k^{'},l^{'})$ $\in \Omega$. Suppose that $\Phi: V_{i,j,k,l}\to V_{i^{'},j^{'},k^{'},l^{'}}$ is an isomorphism of $D$-modules. Let $[\Phi]=(p_{i,j})$$_{i,j=1,2}$ denote the matrix representation of  $\Phi$ with respect to the given basis. Since $[b][\Phi]=[\Phi][b],[c][\Phi]=[\Phi][c],$ we have  $l=l^{'}, (-1)^l p_{12}=p_{21},p_{11}=p_{22}$. Since
		\begin{flalign*}
			[x][\Phi]=[\Phi][x],~~[a][\Phi]=[\Phi][a],~~[t][\Phi]=[\Phi][t],
		\end{flalign*}
		we have
		\begin{gather}
			(\,(-1)^i-(-1)^{i^{'}}\,)p_{11}=0,\quad(\,(-1)^{i+k}-(-1)^{i^{'}}\,)p_{12}=0\label{l1},\\
			(-1)^l(\,(-1)^{i}-(-1)^{i^{'}+k^{'}}\,)p_{12}=0,\quad(\,(-1)^{i+k}-(-1)^{i^{'}+k^{'}}\,)p_{11}=0,\label{l2}\\
			(\xi^k-\xi^{k^{'}})p_{11}=0,\quad((-\xi)^k-\xi^{k^{'}})p_{12}=0,\label{l3}\\
			(-1)^l(\xi^k-(-\xi)^{k^{'}})p_{12}=0,\quad(\,(-\xi)^k-(-\xi)^{k^{'}})p_{11}=0,\label{l4}\\
			(\xi^j-\xi^{{j}^{'}})p_{11}=0,\quad(\,(-1)^{j+k+l}\xi^{j}-\xi^{j^{'}})p_{12}=0,\label{l5}\\
			(-1)^{l}(\xi^{j}-(-1)^{j^{'}+k^{'}+l^{'}}\xi^{j^{'}})p_{12}=0,\label{l6}\\
			(\,(-1)^{j+k+l}\xi^{j}-(-1)^{j^{'}+k^{'}+l^{'}}\xi^{j^{'}})p_{11}=0.\label{l7}
		\end{gather}
		When $l=0$, we have $l^{'}=0$, and $j=j^{'}=0$. By (\ref{l2}) together with (\ref{l4})---(\ref{l7}), we have  $i=i^{'},k=k^{,}$. When $l=1$, $p_{12}=-p_{21}$ implies  $i=i^{'},j=j^{'},k=k^{'}$.  This completes the proof of the claim.
	\end{proof}
	\begin{thm}\label{sims}
		There are 88 simple left $D$-modules up to isomorphism, among which 32 one-dimensional objects are given by Lemma \ref{1simple} and 56 two-dimensional objects are given by Lemma \ref{2simple}.
	\end{thm}
	\begin{proof}
		Suppose that  there exists a simple module of dimension $d>2$ and let $n$ denote the number of isomorphism classes of simple  $d$-dimensional. In view of  Lemmas \ref{1simple} and \ref{2simple}, we have
		\begin{flalign*}
			32\cdot 1^2+56 \cdot 2^2 +nd =256+nd^2\leq \dim (D^*)_{0}=\dim D^*=256,
		\end{flalign*}
		which implies $n=0$. Hence, all simple $D$-modules have dimension at most $2$.
	\end{proof}
	\section{The category $^H_H\mathcal{YD}$}
	
	In this section, by using the monoidal isomorphism
	$^H_H\mathcal{YD}\cong {_D\mathcal M}$, we will describe the simple
	objects in $^H_H\mathcal{YD}$ and determine their braidings.
	
	\begin{pro}\label{1dim}
		Let $\mathbb{K}_{\chi_{i,j,k,l}}=\mathbb{K}\{v\}$ be a one-dimensional $D$-module with $i,~k,~l\in \mathbb{I}_{0,1},~ j\in \mathbb{I}_{0,3}$. Then $\mathbb{K}_{\chi_{i,j,k,l}}\in {^H_H\mathcal{YD}}$ with its module and comodule structure given by
		\begin{equation*}
			x\cdot v=(-1)^iv,~~~y\cdot v=(-1)^jv,~~~t\cdot v=\xi^jv,~~~\delta(v)=x^{j+2k+2l}y^k\otimes v.
		\end{equation*}
	\end{pro}
	 For any $(i, k, l)\in \mathbb{I}_{0,1}$, $j\in \mathbb{I}_{0,3}$, the braiding  on $\mathbb{K}_{\chi_{i,j,k,l}}$ is given by $c(v\otimes v)=(-1)^{ij+jk}v\otimes v$.
	\begin{proof}
		The action is given by the restriction of the action given in Lemma \ref{1simple}. Since $\mathbb{K}_{\chi_{i,j,k,l}}$ is one-dimensional, we must have that $\delta(v)=h\otimes v$ with $h\in G(H)=\langle x\rangle \times  \langle y\rangle.$ As $f \cdot v=\langle f,h\rangle v$ for all $f\in H^{*}$, it follows that $\delta(v)=x^{j+2k+2l}y^k\otimes v$.
		
		By formula (\ref{braiding}), we have
		\begin{flalign*}
			c(v\otimes v)=x^{j+2k+2l}y^k\cdot v\otimes v=(-1)^{ij+jk}v\otimes v.
		\end{flalign*}
		Thus the braiding is clear.
	\end{proof}
	\begin{pro}\label{2dim}
		If $V_{i,j,k,l}=\mathbb{K}\{v_1,v_2\}$ is a two-dimensional simple $D$-module with $(i,j,k,l)\in \Omega$, then $V_{i,j,k,l}\in ^{H}_{H}\mathcal{YD}$ with its action given by
		\label{2dim}
		\begin{gather*}
			x.v_1=(-1)^{i}v_1, \qquad y.v_1=(-1)^{j}v_1, \qquad t.v_1=\xi ^j v_1,\\
			x.v_2=(-1)^{i+k}v_2, \qquad y.v_2=(-1)^{j}v_2, \qquad t.v_2=(-1)^{j+k+l}\xi ^j v_2.
		\end{gather*}
		and its coaction by
		\begin{equation*}
			\begin{split}
				\delta(v_1)&=\frac{1}{2}(x^{k+l}y^{[\frac{k}{2}]}t^{k(mod2)}+x^{k+l+2}y^{[\frac{k}{2}]}t^{k(mod2)})\otimes v_1\\
				&\quad +\frac{1}{2}(-1)^{[\frac{k}{2}]}\xi^{k+l}(x^{k+l}y^{[\frac{k+1}{2}]}t^{k(mod2)}-x^{k+l+2}y^{[\frac{k+1}{2}]}t^{k(mod2)})\otimes v_2,\\
				\delta(v_2)&=\frac{1}{2}(-1)^{[\frac{k}{2}]}\xi^{k-l}(x^{k+l}y^{[\frac{k+1}{2}]}t^{k(mod2)}-x^{k+l+2}y^{[\frac{k+1}{2}]}t^{k(mod2)})\otimes v_1\\
				&\quad+\frac{1}{2}(x^{k+l}y^{[\frac{k+1}{2}]}t^{k(mod2)}+x^{k+l+2}y^{[\frac{k+1}{2}]}t^{k(mod2)})\otimes v_2.\\
			\end{split}
		\end{equation*}	
	\end{pro}		
	\begin{proof}
		It suffices to describe the coaction.
		 Let $\{h_i\}_{1\leq i\leq 16}$ be a basis of $H$, and let $\{h^{i}\}_{1\leq i\leq 16}$ be the dual basis of $H^*$. Then the comodule structure  of $V_{i,j,k,l}$ is given by $\delta(v)=\sum_{i=1}^{16}h_{i}\otimes h^{i}\cdot v$. It follows that
		\begin{flalign*}
			\delta(v_1)&=\sum_{m=0}^{3}\sum_{n=0}^{1}\sum_{r=0}^{1}(a^{m}b^{n}c^{r})^{*}\otimes a^{m}b^{n}c^{r}\cdot v_1\\
			&=(1^*+(-1)^kb^*+\xi ^i a^*+(-1)^k \xi^i (ab)^*+(-1)^i(a^2)^*+(-1)^{i+k}(a^2b)^*+(-\xi)^i(a^3)^*\\
			&\quad +(-1)^{i+k}\xi^{i}(a^3b)^*)\otimes v_1\\
			&\quad +((-1)^kc^*+(bc)^*+(-1)^k\xi ^i(ac)^*+\xi ^i(abc)^*+(-1)^{i+k}(a^2c)^*+(-1)^i(a^2bc)^*\\
			& \quad +(-1)^{i+k}\xi ^i(a^3c)^*+(-\xi)^{*}(a^3bc)^*)\otimes v_2,\\
			\delta(v_2)&=(c^*+(-1)^l(bc)^*+(-1)^k(a^2c)^*+(-\xi)^k(a^3c)^*+(-1)^l(bc)^*+(-1)^l\xi^k(abc)^*\\
			&\quad +(-1)^{l+k}(a^2bc)^*+(-1)^{l+k}\xi ^k(a^3bc)^*)\otimes v_1\\
			&\quad +(1^*+\xi^k(a^3)^*+(-1)^k(a^2)^*+(-\xi)^ka^*+(-1)^lb^*+(-1)^l\xi^k(a^3b)^*\\
			&\quad +(-1)^{l+K}(a^2b)^*+(-1)^{l+k}\xi^k(ab)^*)\otimes v_2.
		\end{flalign*}
		By discussing all  possible situations, the claim follows.
	\end{proof}

	Similarly, we can prove the following Propositions.
	\begin{pro}
		If $W^{1}_{i,j,k}=\mathbb{K}\{v_1,v_2\}$ is a two-dimensional simple $D$-module with $(i,j,k)\in \Lambda^1$, then $W^{1}_{i,j,k}\in ^{H}_{H}\mathcal{YD}$ with its action given by
		\begin{gather*}x.v_1=\xi^iv_1,\qquad y.v_1=(-1)^jv_1,\qquad t.v_1=(-1)^j \xi^i v_2,\\
			x.v_2=(-\xi)^{i}v_2,\qquad y.v_2=(-1)^jv_2,\qquad t.v_2=\xi^i v_1,
		\end{gather*}
		and its coaction by			
		\begin{equation*}
			\begin{split}
				\delta(v_1)&=\frac{1}{2}(x^{i+k}y^{[\frac{i}{2}]}t+x^{i+k+2}y^{[\frac{i}{2}]}t)\otimes v_1\\
				&\quad +\frac{1}{2}(-1)^{[\frac{i}{2}]}\xi^{i+k}(x^{i+k}y^{[\frac{i}{2}]+1}t-x^{i+k+2}y^{[\frac{i}{2}]+1}t)\otimes v_2,\\
				\delta(v_2)&=\frac{1}{2}(-1)^{[\frac{i}{2}]}(-\xi)^{i+k}(x^{i+k}y^{[\frac{i}{2}]}t-x^{i+k+2}y^{[\frac{i}{2}]}t)\otimes v_1\\
				&\quad+\frac{1}{2}(x^{i+k}y^{[\frac{i}{2}]+1}t+x^{i+k+2}y^{[\frac{i}{2}]+1}t)\otimes v_2.\\
			\end{split}
		\end{equation*}			
	\end{pro}
	\begin{pro}
		If $W^{2}_{i,j,k}=\mathbb{K}\{v_1,v_2\}$ is a two-dimensional simple $D$-module with $(i,j,k)\in \Lambda^2$, then $W^{2}_{i,j,k}\in ^{H}_{H}\mathcal{YD}$ with its action given by
		\begin{gather*}
			x.v_1=\xi^iv_1, \qquad y.v_1=(-1)^jv_1, \qquad t.v_1=(-1)^j \xi^i v_2,\\
			x.v_2=(-\xi)^{i}v_2, \qquad y.v_2=(-1)^jv_2, \qquad t.v_2=\xi^i v_1,
		\end{gather*}
		and its coaction by
		\begin{equation*}
			\begin{split}
				\delta(v_1)&=\frac{1}{2}(x^{i+k}y^{[\frac{i}{2}]+1}t+x^{i+k+2}y^{[\frac{i}{2}]+1}t)\otimes v_1\\
				&\quad +\frac{1}{2}(-1)^{[\frac{i}{2}]+1}\xi^{i+k}(x^{i+k}y^{[\frac{i}{2}]}t-x^{i+k+2}y^{[\frac{i}{2}]}t)\otimes v_2,\\
				\delta(v_2)&=\frac{1}{2}(-1)^{[\frac{i+k}{2}]}\xi^{i+k}(x^{i+k}y^{[\frac{i}{2}]+1}t-x^{i+k+2}y^{[\frac{i}{2}]+1}t)\otimes v_1\\
				&\quad+\frac{1}{2}(x^{i+k}y^{[\frac{i}{2}]}t+x^{i+k+2}y^{[\frac{i}{2}]}t)\otimes v_2.\\
			\end{split}
		\end{equation*}	
		
	\end{pro}
	\begin{pro}
		If $W^{3}_{i,j,k}=\mathbb{K}\{v_1,v_2\}$ is a two-dimensional simple $D$-module with $(i,j,k)\in \Lambda^2$, then $W^{3}_{i,j,k}\in ^{H}_{H}\mathcal{YD}$ with its action given by
		\begin{gather*}x.v_1=\xi^iv_1, \qquad y.v_1=(-1)^jv_1, \qquad t.v_1=(-1)^{i+j} \xi^i v_2,\\
			x.v_2=(-\xi)^{i}v_2, \qquad y.v_2=(-1)^jv_2, \qquad t.v_2=(-\xi)^i v_1,
		\end{gather*}
		and its coaction by
		\begin{equation*}
			\begin{split}
				\delta(v_1)&=\frac{1}{2}(x^{i+k}t+x^{i+k+2}t)\otimes v_1 +\frac{1}{2}\xi^{i+k}(x^{i+k}yt-x^{i+k+2}yt)\otimes v_2,\\
				\delta(v_2)&=\frac{1}{2}(-\xi)^{i+k}(x^{i+k}t-x^{i+k+2}t)\otimes v_1+\frac{1}{2}(x^{i+k}yt+x^{i+k+2}yt)\otimes v_2.
			\end{split}
		\end{equation*}	
	\end{pro}
	\begin{pro}
		If $W^{4}_{i,j,k}=\mathbb{K}\{v_1,v_2\}$ is a two-dimensional simple $D$-module with $(i,j,k)\in\Lambda^2$, then $W^{4}_{i,j,k}\in ^{H}_{H}\mathcal{YD}$ with its action given by
		\begin{gather*}x.v_1=\xi^iv_1, \qquad y.v_1=(-1)^jv_1, \qquad t.v_1=(-1)^{i+j} \xi^i v_2,\\
			x.v_2=(-\xi)^{i}v_2, \qquad y.v_2=(-1)^jv_2, \qquad t.v_2=(-\xi)^i v_1,
		\end{gather*}
		and its coaction by
		\begin{flalign*}
			\delta(v_1)=&\frac{1}{2}(x^{i+k}yt+x^{i+k+2}yt)\otimes v_1
			+\frac{1}{2}(-\xi^{i+k})(x^{i+k}t-x^{i+k+2}t)\otimes v_2,\\
			\delta(v_2)=&\frac{1}{2}(-1)^k \xi^{i+k}(x^{i+k}yt-x^{i+k+2}yt)\otimes v_1
			+\frac{1}{2}(x^{i+k}t+x^{i+k+2}t)\otimes v_2.
		\end{flalign*}
	\end{pro}
	
	\begin{pro}
		If $U_{i,j,k,l}=\mathbb{K}\{v_1,v_2\}$ is a two-dimensional simple $D$-module with $(i,j,k,l)\in \Gamma$, then $U_{i,j,k,l}\in ^{H}_{H}\mathcal{YD}$ with its action given by
		\begin{gather*}x.v_1=\xi^iv_1, \qquad y.v_1=(-1)^jv_1, \qquad t.v_1=(-1)^{i+j} v_2,\\
			x.v_2=(-\xi)^{i}v_2, \qquad y.v_2=(-1)^jv_2, \qquad t.v_2=v_1,
		\end{gather*}
		and its coaction by
		$$\delta(v_1)=x^{l+(-1)^j2k}y^k\otimes v_1, \qquad
		\delta(v_2)=x^{2|j-k|-l+4}y^{k+1}\otimes v_2.$$
	\end{pro}

	\section{Nichols Algebras in $^{H}_{H}\mathcal{YD}$}
	\begin{lem}\label{nichols}
		Let $(i,j,k,l)\in \mathbb{I}_{0,1}\times \mathbb{I}_{0,3}\times \mathbb{I}_{0,1}\times \mathbb{I}_{0,1}$. The Nichols algebras $\mathcal{B}(\mathbb{K}_{\chi_{i,j,k,l}})$ associated to $\mathbb{K}_{\chi_{i,j,k,l}}=\mathbb{K}v$ are
		$\mathcal{B}(\mathbb{K}_{\chi_{i,j,k,l}})=\left\{\begin{array}{ll}
			\mathbb{K}[v], ~~~~~~~~~~~~~~~~~~~~~~~~otherwise,\\
			\mathbb{K}[v]/(v^2)=\bigwedge \mathbb{K}_{\chi_{i,j,k,l}}, ~~~if~(i+k)j=1,3.
		\end{array}
		\right.$
	\end{lem}\label{nichols}
	\begin{rem}\label{fnic1}
	 $\mathcal{B}(\mathbb{K}_{\chi_{i,j,k,l}})$ is finite dimensional if and only if  $\mathbb{K}_{\chi_{i,j,k,l}}$ is isomorphic to one of the following Yetter-Drinfeld modules
		\begin{equation*}
			\begin{split}
				V_{1}:&=\mathbb{K}_{\chi_{0,1,1,0}},
				\qquad	V_{2}:=\mathbb{K}_{\chi_{0,1,1,1}},
				\qquad
				V_{3}:=\mathbb{K}_{\chi_{0,3,1,0}},
				\qquad
				V_{4}:=\mathbb{K}_{\chi_{0,3,1,1}},\\
				V_{5}:&=\mathbb{K}_{\chi_{1,1,0,0}},	\qquad	V_{6}:=\mathbb{K}_{\chi_{1,1,0,1}},	\qquad	V_{7}:=\mathbb{K}_{\chi_{1,3,0,0}},	\qquad
				V_{8}:=\mathbb{K}_{\chi_{1,3,0,1}}.
			\end{split}
		\end{equation*}
	\end{rem}		
	\begin{rem}
		For the convenience of our statements, we set
		\begin{equation*}
			\begin{split}
				M_1 &:=V_{0,1,2,0},\qquad M_2:=V_{0,2,1,0},\qquad M_3=V_{0,2,1,1},\qquad M_4:=V_{0,2,3,0},\\
				M_5 &:=V_{0,2,3,1},\qquad M_6:=V_{1,0,0,1},\qquad M_7=V_{1,0,2,1}, \qquad M_8:=V_{1,1,2,0},\\
				M_9 &:=U_{1,0,0,2}, \qquad M_{10}:=U_{1,0,1,0},\quad M_{11}=U_{1,1,0,2}, \quad M_{12}=U_{1,1,1,2}.\\
			\end{split}
		\end{equation*}
	\end{rem}
	\begin{lem}\label{vfnic}
		Let $V_{i,j,k,l}\in ^{H}_{H}\mathcal{YD}$ with $(i,j,k,l)\in \Omega.$ Then
		
		$(1)$ $dim\,\mathcal{B}(V_{i,j,k,l})=\infty,$ for $V_{i,j,k,l}\notin\{M_1,...,M_{8}\}.$
		
		$(2)$ For any $V\in\{M_1,...,M_{8}\}$, $dim \,\mathcal{B}(V)=4$ and
		the relations of $\mathcal{B}(V)$ are given by
		$$
		v_1^2=0,\qquad v_1v_2+v_2v_1=0,\qquad v_2^2=0.
		$$
	\end{lem}
	\begin{proof}
		(1) If $V_{i,j,k,l} \notin \{ M_1,...,M_8\}$, then  the generalized Dynkin diagram associated to the braiding is one of the following
		\begin{center}	
			\begin{tikzpicture}
				\draw (0,0) circle (2pt) node[above] {1};
				\draw (2,0) circle (2pt) node[above] {1};
				\node at (3,0) {or};
				\draw (4,0) circle (2pt) node[above] {$\xi$};
				\draw (4,0) -- (6,0) node[midway, above] {$-1$};
				\draw (6,0) circle (2pt) node[above] {$\xi$};
				\node at (7,0) {or};
				\draw (8,0) circle (2pt) node[above] {$-\xi$};
				\draw (8,0) -- (10,0) node[midway, above] {$-1$};
				\draw (10,0) circle (2pt) node[above] {$-\xi$};
			\end{tikzpicture}.
		\end{center}
		
		In fact, the braiding of $V_{i,j,k,l}$ associated to  the first generalized Dynkin diagram has an eigenvector $v\otimes v$ with eigenvalue $1$, thus the  $\mathcal{B}(V_{i,j,k,l})$
		is infinite dimensional corresponding to it by Remark $\ref{bra}$. The remaining two cases of $(1)$  follow from \cite{H09}.
		
		(2)  For any $V \in \{M_1,...,M_8\}$, their braidings are of Cartan type $A_1\times A_1$. Clearly, we have the required relations of $\mathcal {B}(V)$.
	\end{proof}
	\begin{lem}\label{wfnic}
		Let $W^1_{i,j,k}$, $W^2_{i,j,k}\in ^{H}_{H}\mathcal{YD}$ with $(i,j,k)\in \Lambda^1 $, $W^3_{i,j,k},  W^4_{i,j,k}\in ^{H}_{H}\mathcal{YD}$ with $(i,j,k)\in \Lambda^2$. Then $\dim \mathcal{B}(W^p_{i,j,k})=\infty$, for $p\in\mathbb{I}_{1,4}$, $(i,j,k)\in\Lambda^1 \cup \Lambda^2 $.
	\end{lem}
	\begin{proof}
		By \cite [Section 3.7]{AGi2018}, we have all  the braidings  of $W^p_{i,j,k}$ for all $p\in \mathbb{I}_{1,4}$ and $(i, j, k,) \in \Lambda^1\cup\Lambda^2$ are of diagonal type. Therefore, the assertion is an immediate consequence of \cite{H09}.
	\end{proof}
	\begin{lem}\label{ufnic}
		Let $U_{i,j,k,l}\in ^{H}_{H}\mathcal{YD}$ with $(i,j,k,l)\in \Gamma.$ Then
		
		$(1)$ $dim\,\mathcal{B}(U_{i,j,k,l})=\infty,$ for $U_{i,j,k,l}\notin\{M_9,...,M_{12}\}.$
		
		$(2)$ For any $U\in\{M_9,...,M_{12}\}$, $dim\,\mathcal{B}(U)=4$ and
		the relations of $\mathcal{B}(U)$ are given by
		$$
		v_1^2=0,\qquad v_1v_2+v_2v_1=0,\qquad v_2^2=0.
		$$
	\end{lem}
	\begin{proof}
		The proof is analogous to that of Lemma
		$\ref{vfnic}$, so we omit it here.
	\end{proof}

	Based on Remarks $\ref{fnic1}$ and Lemmas $\ref{vfnic}$,  $\ref{wfnic}$ and $\ref{ufnic}$, we conclude the following Proposition.
	\begin{pro}
		Let $V,W$ be simple objects in $^{H}_{H}\mathcal{YD}$. Then $\mathcal{B}(V\oplus W)$ is finite dimensional and $\mathcal{B}(V\oplus W)\simeq \mathcal{B}(V)\otimes \mathcal{B}(W)$ if and only if $V,W$ satisfy one of the following cases
		
		$(1)~~ V=W=V_{i}, ~i\in  \mathbb{I}_{1,8}$.
		
		$(2)~~ V=M_1, ~or~ M_7, ~ W=V_{i}, ~i\in \mathbb{I}_{1,4}$.
		
		$(3)~~ V=M_8, ~ W=V_{i+4},~ i\in \mathbb{I}_{1,4}$.
		
		$(4)~~ V=M_9, ~or~ M_{10}, ~ W\in \{V_2,V_4,V_5,V_7\}$.
		
		$(5)~~ V=W=M_{i}, i\in \mathbb{I}_{1,12}$.
		
		$(6)~~ V=M_{1}, ~W=M_{6}, ~or~M_{8}$.
		
		$(7)~~ V=M_{i}, ~W=M_{i+2},~i\in \{2,3\}$.
		
		$(8)~~ V=M_{i}, ~W=M_{i+1},~i\in \{6,7,9,11\}$.
		
	\end{pro}
	\begin{proof}
		Since $\mathcal{B}(V \oplus W)$ is finite dimensional, by Remarks $\ref{fnic1}$ and Lemmas $\ref{vfnic}$, $\ref{wfnic}$ and $\ref{ufnic}$, we have  $V,W\in \{ V_{1}, ..., V_8,M_1,...,M_{12}\}$.
		
		Let $V=M_{1}=\mathbb{K}\{p_1,p_2\},$ $W=\mathbb{K}_{\chi_{i,j,k,l}}=\mathbb{K}\{q\}$, we have
		\begin{gather*}
			c(p_1\otimes q)=(-1)^j q\otimes p_1,
			\qquad
			c(p_2\otimes q)=(-1)^jq\otimes p_2,\\
			c(q\otimes p_1)=(-1)^k p_1\otimes q,\qquad c(q\otimes p_2)=(-1)^kp_2\otimes q.
		\end{gather*}
		Then by Lemma  $\ref{nicdim}$, $\mathcal{B}(V\oplus W)\simeq \mathcal {B}(V)\otimes \mathcal {B}(W)$ if and only if $(-1)^{j+k}=1$,~that is, $W=V_i,i\in \mathbb{I}_{1,4}$.
		
		Next,  we can  perform our  discussions  case by case, and  leave the rest to the reader.
	\end{proof}

	\section{Hopf algebras over $H$}
	In this section, based on the principle of the lifting method, we will determine finite-dimensional Hopf algebras over $H$ such that their infinitesimal braidings are those Yetter-Drinfeld modules listed in $(1),(7),(8),(9),(10)$ of \textbf{Theorem  A}. We first show that the diagrams of these Hopf algebras are Nichols algebras.
	\begin{thm}\label{grA}
		Let $A$ be a finite-dimensional Hopf algebras over $H$ such that its infinitesimal braiding is isomorphic to a Yetter-Drinfeld module $N$ as listed in $(1)$, $(7)$, $(8)$, $(9)$, $(10)$ of  \textbf{Theorem A}. Then $gr A\cong \mathcal{B}(N)\sharp H$.
	\end{thm}
	\begin{proof}
		Let $S$ be the graded dual of the diagram  $R$  of $A$. Since  $\mathcal P(R)=R(1)$, by{\rm\cite[Lemma 5.5]{AS0}},  $S$ is generated  by $W=S(1)$  as  algebra. Since  $S$ is a graded Hopf algebra in  $^H_H\mathcal{YD}$ with $S(0)=\mathbb{K}$, by {\rm\cite[Proposition 2.2]{AS2}}, it follows that there exists  a surjective morphism  of graded Hopf algebra  $\phi: S\to \mathcal{B}(W)$. If $\mathcal P(S)=S(1)$, then  $S$ is a  Nichols algebras. Under this condition, by {\rm\cite[Lemma 2.4]{AS2}}, R is also a Nichols algebra. To prove that $\mathcal P(S)=S(1)$, it is suffices to verify that  $\phi$ is also an injective morphism, in other words, the relations of $\mathcal B(W)$ also hold in $S$. We conclude that $W$ is one of the Yetter-Drinfeld modules listed in $(1),(7),(8),(9),(10)$ of \textbf{Theorem  A}.
		
		Assume that $W=\Omega_{6}=M_6\oplus M_6$, it is clear that $W$ is generated by $p_1$, $p_2$, $q_1$, $q_2$, with $M_{6}=\mathbb{K}\{p_1,p_2\}=\mathbb{K}\{q_1,q_2\},$	the defining ideal of the Nichols algebra $\mathcal B(W)$ is genetated by the elements
		\begin{gather*}
			p_1^2,\quad p_2^2,\quad p_1p_2+p_2p_1,\quad q_1^2,\quad q_2^2,\quad q_1q_2+q_2q_1,\\
			p_1q_1+q_1p_1+p_2q_2+q_2p_2,\quad p_1q_2+q_1p_2+p_2q_1+q_2p_1.
		\end{gather*}
		
		By {\rm\cite[Theorem 6]{X}}, all those generators of the defining ideal are primitive elements. We show that $c(r\otimes r)=r\otimes r$ for all generators $r$ given  above for the defining ideal.
		Since
		\begin{equation*}
			\begin{split}
				\delta(p_1)&=\frac{1}{2}(\,(1+x^2)x\otimes p_1+\xi (1-x^2) x\otimes p_2\,),\\
				\delta(p_2)&=\frac{1}{2}(\,(1+x^2)x\otimes p_2-\xi (1-x^2)x\otimes p_1\,),\\
			\end{split}
		\end{equation*}
		we have
		\begin{equation*}
			\begin{split}
				\delta(p_1^2)&=\frac{1}{2}(\,(1+x^2)\otimes p_1^2+ (1-x^2) x\otimes p_2^2\,),\\
				\delta(p_2^2)&=\frac{1}{2}(\,(1+x^2)\otimes p_2^2+ (1-x^2)\otimes p_1^2\,),\\
				\delta(p_1p_2)&=\frac{1}{2}(\,(1+x^2)\otimes p_1p_2- (1-x^2) x\otimes p_2p_1\,),\\
				\delta(p_2p_1)&=\frac{1}{2}(\,(1+x^2)\otimes p_2p_1- (1-x^2)\otimes p_1p_2\,).\\
			\end{split}
		\end{equation*}
		Then by the definition of the braiding in $^H_H\mathcal{YD}$, the claim follows. We leave the rest to the reader.
	\end{proof}
	\begin{lem}$($\cite{AS0}, Lemma 6.1$)$
		Let $H$ be a Hopf algebra, $\psi: H\rightarrow H$ an automorphism of
		Hopf algebra, and $V, ~W$ Yetter-Drinfeld modules over H.
		
		$(1)$ Let $V^\psi$ be the same space as that of the underlying $V$ but with action and
		coaction
		\begin{flalign*}
			&h\cdot_\psi v=\psi(h)\cdot v, ~~~\delta^\psi(v)=(\psi^{-1}\otimes id)\delta(v),~~~h\in H,~v\in V.
		\end{flalign*}
		Then $V^\psi$ is also a Yetter-Drinfeld module over H. If $T: V\rightarrow W$ is a morphism in $_H^H\mathcal{YD}$, so is $T^\psi: V^\psi\rightarrow W^\psi$. Moreover, the braiding
		$c: V^\psi\otimes W^\psi\rightarrow W^\psi\otimes V^\psi$ coincides with the braiding $c: V\otimes W\rightarrow W\otimes V$.
		
		$(2)$ If $R$ is an algebra (resp. a coalgebra, a Hopf algebra) in
		$_H^H\mathcal{YD}$, so is $R^\psi$, with the same structural maps.
		
		$(3)$ Let $R$ be a Hopf algebra in $_H^H\mathcal{YD}$. Then the map
		$\varphi: R^\psi\sharp H\rightarrow R\sharp H$ given by $\varphi(r\sharp h)=r\sharp \psi(h)$ is an isomorphism of Hopf algebras.
	\end{lem}

	Based on the observation of Andruskiewitsch and Schneider {\rm\cite [Lemma 6.1]{AS0}}: The bosonization of a Nichols algebra is isomorphic to the bosonization of a twist of its Nichols algebra (with respect to a certain automorphism of the Hopf algebra), we first need to characterize the isomorphism relationships between the  simple Yetter-Drinfeld modules and their twists with respect to automorphisms of  $H$.

	\begin{cor}
		There exist the following isomorphic relationships among the simple Yetter-Drinfeld modules, that is, there exist $\tau_{i}\in\text{Aut}(H)$, such that
		\begin{equation*}
			\begin{split}
				(1)&\quad	 V_1^{\tau_{17}}\cong V_2, \quad  V_3^{\tau_{17}}\cong V_4,\\
				(2)&\quad  V_1^{\tau_{49}}\cong V_5,\quad V_2^{\tau_{49}}\cong V_6,\quad V_3^{\tau_{49}}\cong V_7,\quad V_4^{\tau_{49}}\cong V_8,\\
				(3)&\quad  V_2^{\tau_{33}}\cong V_5,\quad \  V_4^{\tau_{33}}\cong V_4,\\
				(4)& \quad M_1^{\tau_{49}}\cong M_8, \quad M_2^{\tau_{49}}\cong M_4,\quad M_3^{\tau_{55}}\cong M_5,\\
				(5)& \quad M_6^{\tau_{49}}\cong M_7,\quad M_9^{\tau_{12}}\cong M_{11}, \quad M_{10}^{\tau_{12}}\cong M_{12},\\
				(6)&\quad \mathcal{B}(\Omega_1)\sharp H\cong \mathcal{B}(\Omega_8)\sharp H, \quad \mathcal{B}(\Omega_6)\sharp H\cong \mathcal{B}(\Omega_7)\sharp H,\\
				(7)&\quad \mathcal{B}(\Omega_i)\sharp H\cong \mathcal{B}(\Omega_{i+2})\sharp H,  \quad i\in\{2,3,9,10\},\\
				(8)&\quad \mathcal{B}(\Omega_6^{(1)})\sharp H\cong \mathcal{B}(\Omega_7^{(3)})\sharp H, 	\quad \mathcal{B}(\Omega_9^{(3)})\sharp H\cong \mathcal{B}(\Omega_{11}^{(3)})\sharp H.
			\end{split}
		\end{equation*}
	\end{cor}

	\begin{defi}\label{def1}
		Denote by $\mathfrak{U}_{1,1}(\lambda)$ the algebra  generated by $x, y, t, p, q$ satisfying   relations (\ref{3.1}) and
		\begin{gather}
			xp=px,\quad yp=-py,\quad tp=\xi px^2t,\quad xq=qx,\quad yq=-qy,\quad tq=\xi qx^2t	\label{u111},\\
			p^2=q^2=\lambda(1-x^2),\quad pq+qp=\mu(1-x^2).
		\end{gather}
		\label{u112}
		It is a Hopf algebra with its coalgebra structure determined by (\ref{3.2}) and
		\begin{gather}\label{u113}
			\Delta(p)=p\otimes 1+x^3y \otimes p,\quad \Delta(q)=q\otimes 1+x^3y\otimes q.\end{gather}
	\end{defi}
	\begin{defi}\label{def2}
		Denote by $\mathfrak{U}_{1,2}(\lambda)$ the algebra  generated by $x,y,t,p,q$ satisfying   relations (\ref{3.1}), (\ref{u111}) and
		\begin{gather}\label{u121}
			p^2=q^2=\lambda(1-x^2),\quad pq=qp=0.
		\end{gather}
		It is a Hopf algebra with its coalgebra structure determined by (\ref{3.2}) and (\ref{u113}).
	\end{defi}\label{def3}
	\begin{defi}
		Denote by $\mathfrak{U}_{1,3}(\lambda)$ the algebra generated by $x, y, t, p, q$ satisfying  relations (\ref{3.1}), (\ref{u121}) and
		\begin{gather}
			xp=px,\ yp=-py,\ tp=\xi px^2t,\
			xq=qx,\ yq=-qy,\ tq=-\xi qx^2t.\label{u131}
		\end{gather}
		It is a Hopf algebra with its coalgebra structure determined by (\ref{3.2}) and
		\begin{center}$\Delta(p)=p\otimes 1+x^3y \otimes p,\hspace{1em}\Delta(q)=q\otimes 1+xy\otimes q$.
		\end{center}		
	\end{defi}
	
	\begin{defi}\label{def4}
		Denote by $\mathfrak{U}_{1,4}(\lambda)$ the algebra  generated by $x,y,t,p,q$ satisfying  (\ref{3.1}), (\ref{u131})  and
		\begin{gather*}
			p^2=q^2=\lambda(1-x^2),\ pq=qp=\mu(1-x^2).
		\end{gather*}
		It is a Hopf algebra with its coalgebra structure determined by (\ref{3.2}) and (\ref{u113}).					
	\end{defi}
	\begin{defi}\label{def5}
		Denote by $\mathfrak{U}_{1,5}(\lambda,\mu)$ the algebra  generated by $x,y,t,p,q$ satisfyingthe relations (\ref{3.1})  and
		\begin{gather}
			xp=px,\ yp=-py,\ tp=\xi px^2t,\
			xq=-qx,\ yq=-qy,\ tq=\xi qx^2t, \label{u151}\\
			p^2=q^2=\lambda(1-x^2),\ pq=qp=\mu(1-y).\label{u152}
		\end{gather}	
		It is a Hopf algebra with its coalgebra structure determined by (\ref{3.2}) and
		\begin{gather}\label{u153}
			\Delta(p)=p\otimes 1+x^3y \otimes p,\quad\Delta(q)=q\otimes 1+x\otimes q.
		\end{gather}
	\end{defi}
	\begin{defi}\label{def6}
		Denote by $\mathfrak{U}_{1,6}(\lambda,\mu)$ the algebra  generated by $x,y,t,p,q$ satisfying  relations (\ref{3.1}), (\ref{u151}) and
		\begin{gather}\label{u161}
			p^2=q^2=\lambda(1-x^2),\hspace{1em}pq=qp=\mu(1-x^2y).
		\end{gather}
		It is a Hopf algebra with its coalgebra structure determined by (\ref{3.2}) and
		\begin{gather}\label{163}
			\Delta(p)=p\otimes 1+x^3y \otimes p,\quad\Delta(q)=q\otimes 1+x^3\otimes q.
		\end{gather}			
	\end{defi}
	\begin{defi}\label{def7}
		Denote by $\mathfrak{U}_{1,7}(\lambda,\mu)$ the algebra  generated by $x, y, t, p, q$ satisfying  relations (\ref{3.1}), (\ref{u161}) and
		\begin{gather}
			xp=px,\quad yp=-py, \quad tp=\xi px^2t,\quad
			xq=-qx,\quad yq=-qy,\quad tq=-\xi qx^2t.\label{u171}
		\end{gather}
		It is a Hopf algebra with its coalgebra structure determined by (\ref{3.2}) and
		\begin{center}$\Delta(p)=p\otimes 1+x^3y \otimes p,\quad \Delta(q)=q\otimes 1+x^3\otimes q.$
		\end{center}				
	\end{defi}
	\begin{defi}\label{def8}
		Denote by $\mathfrak{U}_{1,8}(\lambda,\mu)$ the algebra  generated by $x,y,t,p,q$ satisfying  relations (\ref{3.1}), (\ref{u171}) and (\ref{u152}).
		It is a Hopf algebra with its coalgebra structure determined by (\ref{3.2}) and (\ref{u153}).
		\begin{pro}
			Assume that  $A$ is a finite-dimensional Hopf algebras with the coradical $H$ such that its infinitesimal braiding is isomorphic to $\Omega_{1,i}$ with $i\in \mathbb{I}_{1,8}$, then $A\cong \mathfrak{U}_{1,i}(\lambda,\mu).$
		\end{pro}
		\begin{proof}
			We prove that claim for $\Omega_{1,1}.$ The proofs for  the remaining cases $\Omega_{1,i}$ follow by the  same argument. By Theorem $\ref{grA}$, there exits an isomorphism of  Hopf algebra $gr(A)\cong \mathcal(\Omega_{1,1})\sharp H$. It suffices to check the relations (6.2) listed in Definition $\ref{def1}$  hold in $A$. Let $V_{1}=\mathbb{K}\{p\},V_2=\mathbb\{q\}$. A direct verification shows that
			\begin{flalign*}
				&\Delta(p^2)=p^2\otimes 1+x^2\otimes p^2,~~~~
				\Delta(q^2)=q^2\otimes 1+x^2\otimes q^2,\\
				&\Delta(pq)=pq\otimes 1+x^2\otimes pq,~~~~
				\Delta(qp)=qp\otimes 1+x^2\otimes qp.
			\end{flalign*}
			Thus $p^2=q^2=\lambda(1-x^2)$ and $pq=qp=\mu(1-x^2)$ for some $\lambda,\mu \in \mathbb{K}$. Consequently there exists  a surjective Hopf algebra homomorphism from $\mathfrak{U}_{1,1}(\lambda ,\mu)$ to $A$. Moreover, every element of $\mathfrak{U}_{1,1}(\lambda,\mu)$  is a linear combination of
			$$
			\{\,p^iq^jx^ky^lt^m\mid i,j,l,m\in \mathbb{I}_{0,1},k\in \mathbb{I}_{0,3}\,\}.
			$$
			  Applying the Diamond Lemma (\cite{G}), this set forms a basis of $\mathfrak{U}_{1,1}(\lambda,\mu)$. Hence $\dim A = \dim \mathfrak{U}_{1,1}(\lambda,\mu)$, and the surjection is an isomorphism. Therefore, $A\cong \mathfrak{U}_{1,1}(\lambda,\mu)$.
		\end{proof}		
	\end{defi}
	
	Let $I=\{\lambda,\mu,\alpha\}$, $J=\{\lambda, \mu,\alpha,\beta,\gamma,\eta \}$ be  sets of parameters, where $\lambda, \mu, \alpha,\beta,\gamma,\eta\in \mathbb{K}$.
	\begin{defi}\label{def9}
		Denote by $\mathfrak{U}_{1}(\lambda,\mu)$ the algebra generated by $x,y,t,p_{1},p_{2},q_{1},q_{2}$ satisfying relations (\ref{3.1}), and
		\begin{gather}
			xp_{1}=p_{1}x,\ yp_{1}=-p_{1}y,\
			tp_{1}=\xi p_{1}x^2t,\
			xp_{2}=p_{2}x,\ yp_{2}=-p_{2}y,\
			tp_{2}=-\xi p_{2}x^2t,\\
			xq_{1}=q_{1}x,\ yq_{1}=-q_{1}y,\
			tq_{1}=\xi q_{1}x^2t,\ xq_{2}=q_{2}x,\ yq_{2}=-q_{2}y,\
			tq_{2}=-\xi q_{2}x^2t,	\\
			p_{1}^2-p_2^2=\lambda(1-x^2),\quad p_{1}p_{2}+p_{2}p_1=0,\\	
			q_{1}^2-q_2^2=\mu(1-x^2),\quad q_{1}q_{2}+q_{2}q_1=0,\\
			p_1q_1+q_1p_1+p_2q_2+q_2p_2=0,\\
			p_1q_2+q_1p_2+p_2q_1+q_2p_1=0.
		\end{gather}
		
		It is a Hopf algebra with its coalgebra structure determined by (\ref{3.2}) and
		\begin{gather*}
			\Delta(p_1)=p_1\otimes 1+\frac{1}{2}(\,(1+x^2)y\otimes p_1-(1-x^2)y\otimes p_2\,),\\
			\Delta(p_2)=p_2\otimes 1+\frac{1}{2}(\,(1+x^2)y\otimes p_2-(1-x^2)y\otimes p_1\,),\\
			\Delta(q_1)=q_1\otimes 1+\frac{1}{2}(\,(1+x^2)y\otimes q_1-(1-x^2)y\otimes q_2\,),\\
			\Delta(q_2)=q_2\otimes 1+\frac{1}{2}(\,(1+x^2)y\otimes q_2-(1-x^2)y\otimes q_1\,).
		\end{gather*}
		\begin{pro}
			Assume that  $A$ is a finite-dimensional Hopf algebra with  the coradical $H$  such that its infinitesimal braiding is $\Omega_{1}$, then $A\cong \mathfrak{B}(\Omega_{1})\sharp H$.
		\end{pro}
		\begin{proof}
			Let $M_{1}=\mathbb{K}\{p_1,p_2\}, M_{2}=\mathbb{K}\{q_1,q_2\}$. By Theorem $\ref{grA}$, we have that $gr(A)\cong \mathfrak{B}(\Omega_{1})\sharp H$. Recall that $\mathfrak{B}(\Omega_{1})\sharp H$ is the algebra generated by $p_1, p_2, q_1, q_2, x, y, t$ with $p_1, p_2, q_1, q_2$ satisfying the relations of $\mathfrak{B}(\Omega_{1})$, $x,y,t$ satisfying the relations of $H$, and all together satisfying the relations:
			\begin{equation}
				\begin{split}
					xp_1&=p_1x,\ yp_1=-p_1y,\ tp_1=\xi p_1x^2t,\  xp_2=p_2x,\ yp_2=-p_2y,\ tp_2=-\xi p_2x^2t,	\\
					xq_1&=q_1x,\  yq_1=-q_1y,\ tq_1=\xi q_1x^2t,\ xq_2=q_2x,\ yq_2=-q_2y,\ tq_2=-\xi q_2x^2t.
				\end{split}
			\end{equation}
			After a direct computation, we  konw that
			\begin{gather*}
				\Delta(p_1^{2}+p_2^{2})=(p_1^{2}+p_2^{2})\otimes 1 +1\otimes (p_1^{2}+p_2^{2}),\\
				\Delta(p_1^{2}-p_2^{2})=(p_1^{2}-p_2^{2})\otimes 1 +1\otimes (p_1^{2}-p_2^{2}).
			\end{gather*}
			Then $p_1^{2}=\lambda(1-x^2)$, $p_2^{2}=-\lambda(1-x^2)$ for some $\lambda\in \mathbb{K}$. Since $tp_{1}^{2}=-p_{1}^{2}t$, the relations $p_{1}^{2}=p_2^{2}=0$ hold in $A$. Similary, the  relations $p_1p_2+p_2p_1=0,~q_1^{2}=q_2^{2}=0,~ q_1q_2+q_2q_1=0,~p_1q_1+q_1p_1+p_2q_2+q_2p_2=0,~p_1q_2+q_2p_1+p_2q_1+q_1p_2=0.$
		\end{proof}

	\end{defi}\label{def10}
	
	\begin{defi}
		Denote by $\mathfrak{U}_{2}(J)$ the algebra  generated by $x,y,t,p_{1},p_{2},q_{1},q_{2}$ satisfying relations (\ref{3.1}) , and
		\begin{gather}
			xp_{1}=p_{1}x,\ yp_{1}=p_{1}y,\
			tp_{1}=-p_1t,\
			xp_{2}=-p_{2}x,\ yp_{2}=p_{2}y,\
			tp_{2}= p_{2}t,\\
			xq_{1}=q_{1}x,\ yq_{1}=q_{1}y,\
			tq_{1}= -q_{1}t,\ xq_{2}=-q_{2}x,\ yq_{2}=q_{2}y,\
			tq_{2}=q_{2}t,\\
			p_{1}^2=\lambda(1-y)+\mu(1-x^2y), \,p_2^2=\lambda(1-y)-\mu(1-x^2y),\ p_{1}p_{2}+p_{2}p_1=0,\label{u21}\\
			q_{1}^2=\alpha(1-y)+\beta(1-x^2y),\,q_2^2=\alpha(1-y)-\beta(1-x^2y),\ q_{1}q_{2}+q_{2}q_1=0, \label{u22}\\
			p_1q_1+q_1p_1=\gamma(1-y)+\eta(1-x^2y), \,p_2q_2+q_2p_2=\gamma(1-y)-\eta(1-x^2y), \label{u23}\\
			p_1q_2+q_2p_1+p_2q_1+q_1p_2=0. \label{u24}
		\end{gather}
		It is a Hopf algebra with its coalgebra structure determined by (\ref{3.2}) and
		\begin{equation*}
			\begin{split}
				\Delta(p_1)&=p_1\otimes 1+\frac{1}{2}(\,(1+x^2)xt\otimes p_1+\xi(1-x^2)xyt\otimes p_2\,),\\
				\Delta(p_2)&=p_2\otimes 1+\frac{1}{2}(\,\xi(1-x^2)xyt\otimes p_1+(1+x^2)xyt\otimes p_2\,),\\
				\Delta(q_1)&=q_1\otimes 1+\frac{1}{2}(\,(1+x^2)xt\otimes q_1+\xi(1-x^2)xyt\otimes q_2\,),\\
				\Delta(q_2)&=q_2\otimes 1+\frac{1}{2}(\,\xi(1-x^2)xyt\otimes q_1+(1+x^2)xyt\otimes q_2\,).
			\end{split}
		\end{equation*}
	\end{defi}
	\begin{defi}\label{def11}
		Denote by $\mathfrak{U}_{3}(I)$ the algebra  generated by $x,y,t,p_{1},p_{2},q_{1},q_{2}$ satisfying relations (\ref{3.1}), and
		\begin{gather}
			xp_{1}=p_{1}x,\ yp_{1}=p_{1}y,\
			tp_{1}=- p_{1}t,\
			xp_{2}=-p_{2}x,\ yp_{2}=p_{2}y,\
			tp_{2}=-p_{2}t,\\
			xq_{1}=q_{1}x,\ yq_{1}=q_{1}y,\
			tq_{1}= -q_{1}t,\ xq_{2}=-q_{2}x,\ yq_{2}=q_{2}y,\
			tq_{2}=-q_{2}t,\\
			p_{1}^2=p_2^2=\lambda(1-y),\ p_{1}p_{2}+p_{2}p_1=0,\\	
			q_{1}^2=q_2^2=\mu(1-y),\ q_{1}q_{2}+q_{2}q_1=0,\\
			p_1q_1+q_1p_1+p_2q_2+q_2p_2=\alpha(1-x^2y),\\
			p_1q_2+q_1p_2+p_2q_1+q_2p_1=0.
		\end{gather}
		It is a Hopf algebra with its coalgebra structure determined by (\ref{3.2})  and
		\begin{equation*}
			\begin{split}
				\Delta(p_1)&=p_1\otimes 1+\frac{1}{2}(\,(1+x^2)t\otimes p_1+(1-x^2)yt\otimes p_2\,),\\
				\Delta(p_2)&=p_2\otimes 1+\frac{1}{2}(\,(1+x^2)yt\otimes p_2-(1-x^2)yt\otimes p_1\,),\\
				\Delta(q_1)&=q_1\otimes 1+\frac{1}{2}(\,(1+x^2)t\otimes q_1+(1-x^2)yt\otimes q_2\,),\\
				\Delta(q_2)&=q_2\otimes 1+\frac{1}{2}(\,(1+x^2)yt\otimes q_2-(1-x^2)yt\otimes q_1\,).
			\end{split}
		\end{equation*}
	\end{defi}
	\begin{defi}\label{def12}
		Denote by $\mathfrak{U}_{6}(I)$ the algebra generated by $x,y,t,p_{1},p_{2},q_{1},q_{2}$ satisfying relations (\ref{3.1}), and
		\begin{gather}
			xp_{1}=-p_{1}x,\
			yp_{1}=p_{1}y,\
			tp_{1}=p_{1}t,\
			xp_{2}=-p_{2}x,\
			yp_{2}=p_{2}y,\
			tp_{2}=-p_{2}t,\\
			xq_{1}=-q_{1}x,\
			yq_{1}=q_{1}y,\
			tq_{1}=q_{1}t,\
			xq_{2}=-q_{2}x,\
			yq_{2}=q_{2}y,\
			tq_{2}=-q_{2}t,	\\
			p_{1}^2=p_2^2=0,\
			p_{1}p_{2}+p_{2}p_1=\lambda(1-x^2),\\
			q_{1}^2=q_2^2=0,\
			q_{1}q_{2}+q_{2}q_1=\mu(1-x^2),\\
			p_1q_1+q_1p_1+p_2q_2+q_2p_2=0,\\
			p_1q_2+q_1p_2+p_2q_1+q_2p_1=\alpha(1-x^2).
		\end{gather}
		It is a Hopf algebra with its coalgebra structure determined by (\ref{3.2})  and
		\begin{equation*}
			\begin{split}
				\Delta(p_1)&=p_1\otimes 1+\frac{1}{2}(\,(1+x^2)x\otimes p_1+\xi (1-x^2) x\otimes p_2\,),\\
				\Delta(p_2)&=p_2\otimes 1+\frac{1}{2}(\,(1+x^2)x\otimes p_2-\xi (1-x^2)x\otimes p_1\,),\\
				\Delta(q_1)&=q_1\otimes 1+\frac{1}{2}(\,(1+x^2)x\otimes q_1+\xi (1-x^2) x\otimes q_2\,),\\
				\Delta(q_2)&=q_2\otimes 1+\frac{1}{2}(\,(1+x^2)x\otimes q_2-\xi (1-x^2)x\otimes q_1\,).
			\end{split}
		\end{equation*}
	\end{defi}

	\begin{defi}\label{def13}
		Denote by $\mathfrak{U}_{9}(I)$ the algebra  generated by $x,y,t,p_{1},p_{2},q_{1},q_{2}$ satisfying relations (\ref{3.1}), and
		\begin{gather}
			xp_{1}=\xi p_{1}x,\
			yp_{1}=p_{1}y,\
			tp_{1}=-p_{2}t,\
			xp_{2}=-\xi p_{2}x,\
			yp_{2}=p_{2}y,\
			tp_{2}=p_{1}t,\\
			xq_{1}=\xi q_{1}x,\
			yq_{1}=q_{1}y,\
			tq_{1}=-q_{2}t,\
			xq_{2}=-\xi q_{2}x,\
			yq_{2}=q_{2}y,\
			tq_{2}=q_{2}t,\\
			p_{1}^2=p_2^2=0,\
			p_{1}p_{2}+p_{2}p_1=\lambda(1-y),\\	
			q_{1}^2=q_2^2=0,\
			q_{1}q_{2}+q_{2}q_1=\mu(1-y),\\
			p_1q_1+q_1p_1+p_2q_2+q_2p_2=0,\\
			p_1q_2+q_1p_2+p_2q_1+q_2p_1=\alpha(1-y).
		\end{gather}
		It is a Hopf algebra with its coalgebra structure determined by (\ref{3.2})  and
		\begin{equation*}
			\begin{split}
				\Delta(p_1)&=p_1\otimes 1+x^2\otimes p_1,\\
				\Delta(p_2)&=p_2\otimes 1+x^2y\otimes p_2,\\
				\Delta(q_1)&=q_1\otimes 1+x^2\otimes q_1,\\
				\Delta(q_2)&=q_2\otimes 1+x^2y\otimes q_2.
			\end{split}
		\end{equation*}
	\end{defi}
	\begin{defi}\label{def14}
		Denote by $\mathfrak{U}_{10}(I)$ the algebra  generated by $x,y,t,p_{1},p_{2},q_{1},q_{2}$ satisfying relations (\ref{3.1}), and
		\begin{gather}
			xp_{1}=\xi p_{1}x,\
			yp_{1}=p_{1}y,\
			tp_{1}=-p_{2}t,\
			xp_{2}=-\xi p_{2}x,\
			yp_{2}=p_{2}y,\
			tp_{2}=-p_{1}t,\\
			xq_{1}= \xi q_{1}x,\
			yq_{1}=q_1y,\
			tq_{1}=-q_{2}t,\
			xq_{2}=-\xi q_{2}x,\
			yq_{2}=q_{2}y,\
			tq_{2}=-q_{1}t,\\
			p_{1}^2=p_2^2=0,\
			p_{1}p_{2}+p_{2}p_1=\lambda(1-y),\\	
			q_{1}^2=q_2^2=0,\
			q_{1}q_{2}+q_{2}q_1=\mu(1-y),\\
			p_1q_1+q_1p_1+p_2q_2+q_2p_2=0,\\
			p_1q_2+q_1p_2+p_2q_1+q_2p_1=\alpha(1-y).
		\end{gather}
		It is a Hopf algebra with its coalgebra structure determined by (\ref{3.2})  and
		\begin{gather*}
			\Delta(p_1)=p_1\otimes 1+x^2y\otimes p_1,\\
			\Delta(p_2)=p_2\otimes 1+x^2\otimes p_2,\\
			\Delta(q_1)=q_1\otimes 1+x^2y\otimes q_1,\\
			\Delta(q_2)=q_2\otimes 1+x^2\otimes q_2.
		\end{gather*}
	\end{defi}
	\begin{pro}\label{6u2}
		$(1)$	Assume that $A$ is a finite-dimensional Hopf algebra with the coradical $H$ such that the infinitesimal braiding is isomorphic to $\Omega_{i}$ for some $i\in \{2,3,6,9,10\}$, then $A\cong \mathfrak{U}_{i}(I)$.
		
		$(2)$ $\mathfrak{U}_{2}(I)\cong \mathfrak{U}_{2}(I^{'})$ if and only if there exist nonzero parameters $a_1$, $a_2$, $b_1$, $b_2$ such that
		\begin{equation}
			\begin{split}
				\label{(var7)}
				a_1^2\lambda^{'}+a_1a_2\gamma^{'}+a_2^2\alpha^{'}&=\lambda,\\
				a_1^2\mu^{'}+a_1a_2\eta^{'}+a_2^2\beta^{'}&=\mu,\\
				b_1^2\lambda^{'}+b_1b_2\gamma^{'}+b_2^2\alpha^{'}&=\alpha,\\	
				b_1^{2}\mu^{'}+b_1b_2\eta^{'}+b_2^2\beta^{'}&=\beta,\\	2a_1b_1\lambda^{'}+a_1b_2\gamma^{'}+a_2b_1\gamma^{'}+2a_2b_2\alpha^{'}&=\gamma,\\
				2a_1b_1\mu^{'}+a_1b_2\eta^{'}+a_2b_1\eta^{'}+2a_2b_2\beta^{'}&=\eta,
			\end{split}
		\end{equation}
		or satisfying
		\begin{equation}
			\begin{split}
				\label{(var8)}
				a_1^{2}\lambda^{'}&+a_2^{2}\alpha^{'}=\lambda,\\
				a_1^{2}\mu^{'}&+a_2^{2}\beta^{'}=\mu,\\
				b_1^{2}\lambda^{'}&+b_2^{2}\alpha^{'}=\alpha,\\
				b_1^{2}\mu^{'}&+a_2^{2}\beta^{'}=\beta.\\
			\end{split}
		\end{equation}
		$(3)$ $\mathfrak{U}_{3}(I)\cong \mathfrak{U}_{3}(I^{'})$ if and only if there exist nonzero parameters $a_1$, $a_2$, $b_1$, $b_2$ such that
		\begin{equation}\label{(var6)}
			\begin{split}
				a_1^2\lambda^{'}+a_1a_2\alpha^{'}+a_2^2\mu^{'}&=\lambda,\\
				b_1^2\lambda^{'}+b_1b_2\alpha^{'}+b_2^2\mu^{'}&=\mu,\\		2a_1b_1\lambda^{'}+a_1b_2\alpha^{'}+a_2b_1\alpha^{'}+2a_2b_2\mu^{'}&=\alpha.
			\end{split}
		\end{equation}
		$(4)$ $\mathfrak{U}_{6}(I)\cong \mathfrak{U}_{6}(I^{'})$ if and only if there exist nonzero parameters $a_1,a_2$ such that
		\begin{gather}\label{(var4.1)}
			a_1^{2}\lambda^{'}=\lambda,\ a_2^{2}\mu^{'}=\mu,
		\end{gather}
		or satisfying
		\begin{gather}\label{(var4.2)}
			a_1^{2}\mu^{'}=\lambda,\ a_2^{2}\lambda^{'}=\mu,
		\end{gather}
		or satisfying
		\begin{gather} \label{(var4.3)}
			a_1^{2}\lambda^{'}=-\lambda,\ a_2^{2}\mu^{'}=-\mu,
		\end{gather}
		or satisfying
		\begin{gather} \label{(var4.4)}
			a_1^{2}\mu^{'}=-\lambda,\ a_2^{2}\lambda^{'}=-\mu.
		\end{gather}
		$(5)$ For $i\in\{9,10\}$, $\mathfrak{U}_{i}(I)\cong \mathfrak{U}_{i}(I^{'})$ if and only if there exist nonzero parameters $a_1$, $a_2$, $b_1$, $b_2$ such that
		\begin{equation}\label{(var2)}
			\begin{split}
				a_1^2\lambda^{'}+a_1a_2(\alpha^{'}+\beta^{'})+a_2^2\mu^{'}&=\lambda,\\
				b_1^2\lambda^{'}+b_1b_2(\alpha^{'}+\beta^{'})+b_2^2\mu^{'}&=\mu,\\		a_1b_1\lambda^{'}+a_1b_2\alpha^{'}+a_2b_1\beta^{'}+a_2b_2\mu^{'}&=\alpha,\\	
				a_1b_1\lambda^{'}+a_1b_2\beta^{'}+a_2b_1\alpha^{'}+a_2b_2\mu^{'}&=\beta,
			\end{split}
		\end{equation}
		or satisfying
		\begin{equation}\label{(var3)}
			\begin{split}
				-\xi( a_1^2\lambda^{'}+a_1a_2(\alpha^{'}+\beta^{'})+a_2^2\mu^{'})&=\lambda,\\
				-\xi (b_1^2\lambda^{'}+ b_1b_2(\alpha^{'}+\beta^{'})+ b_2^2\mu^{'})&=\mu,\\		-\xi(a_1b_1\lambda^{'}+a_1b_2\alpha^{'}+a_2b_1\beta^{'}+a_2b_2\mu^{'})&=\alpha,\\	
				-\xi(a_1b_1\lambda^{'}+a_1b_2\beta^{'}+a_2b_1\alpha^{'}+a_2b_2\mu^{'})&=\beta,
			\end{split}
		\end{equation}
		or satisfying
		\begin{equation}\label{(var4)}
			\begin{split}
				-( a_1^2\lambda^{'}+a_1a_2(\alpha^{'}+\beta^{'})+a_2^2\mu^{'})&=\lambda,\\
				-(b_1^2\lambda^{'}+ b_1b_2(\alpha^{'}+\beta^{'})+ b_2^2\mu^{'})&=\mu,\\		-(a_1b_1\lambda^{'}+a_1b_2\alpha^{'}+a_2b_1\beta^{'}+a_2b_2\mu^{'})&=\alpha,\\	
				-(a_1b_1\lambda^{'}+a_1b_2\beta^{'}+a_2b_1\alpha^{'}+a_2b_2\mu^{'})&=\beta,
			\end{split}
		\end{equation}
		or satisfying
		\begin{equation}\label{(var5)}
			\begin{split}
				\xi( a_1^2\lambda^{'}+a_1a_2(\alpha^{'}+\beta^{'})+a_2^2\mu^{'})&=\lambda,\\
				\xi(b_1^2\lambda^{'}+ b_1b_2(\alpha^{'}+\beta^{'})+ b_2^2\mu^{'})&=\mu,\\		\xi(a_1b_1\lambda^{'}+a_1b_2\alpha^{'}+a_2b_1\beta^{'}+a_2b_2\mu^{'})&=\alpha,\\	
				\xi(a_1b_1\lambda^{'}+a_1b_2\beta^{'}+a_2b_1\alpha^{'}+a_2b_2\mu^{'})&=\beta.
			\end{split}
		\end{equation}
	\end{pro}
	\begin{proof}
		For $(1)$, We prove the claim for $\Omega_2$. The proofs for others follow the same lines. Let $M_2=\mathbb{K}\{p_1, p_2\}=\mathbb{K}\{q_1, q_2\}$. By Theorem $\ref{grA}$, there exists a Hopf algebra isomorphism $gr A\cong \mathcal{B}(\Omega_2)\sharp H$. Recall that $\mathcal{B}(\Omega_2)\sharp H$ is the algebra generated by $p_1$, $p_2$, $q_1$, $q_2$, $x$, $y$, $t$ with $p_1$, $p_2$, $q_1$, $q_2$, satisfying the relations $\mathcal{B}(\Omega_2)$, $x$, $y$, $t$ satisfying the relations of $H$, and all together satisfying the relations:
			\begin{gather*}
			xp_{1}=p_{1}x,\ yp_{1}=p_{1}y,\
			tp_{1}=-p_1t,\
			xp_{2}=-p_{2}x,\ yp_{2}=p_{2}y,\
			tp_{2}= p_{2}t,\\
			xq_{1}=q_{1}x,\ yq_{1}=q_{1}y,\
			tq_{1}= -q_{1}t,\ xq_{2}=-q_{2}x,\ yq_{2}=q_{2}y,\
			tq_{2}=q_{2}t,
		\end{gather*}
		After a direct computation, we have
		\begin{equation*}
			\begin{split}
			\Delta(p_1^2+p_2^2)&=(p_1^2+p_2^2)\otimes 1+y\otimes (p_1^2+p_2^2),\\
				\Delta(p_1^2-p_2^2)&=(p_1^2-p_2^2)\otimes 1+x^2y\otimes (p_1^2-p_2^2),\\	
					\Delta(p_1p_2+p_2p_1)&=(p_1p_2+p_2p_1)\otimes 1+\frac{1}{2}(1+y+x^2-x^2y)\otimes(p_1p_2+p_2p_1),\\	
			\Delta(p_1q_1{+}q_1p_1{+}p_2q_2{+}q_2p_2)&=(p_1q_1{+}q_1p_1{+}p_2q_2{+}q_2p_2)\otimes 1\\
			&\quad+y\otimes(p_1q_1{+}q_1p_1{+}p_2q_2{+}q_2p_2),\\
			\Delta(p_1q_1{+}q_1p_1{-}p_2q_2{-}q_2p_2)&=(p_1q_1{+}q_1p_1{-}p_2q_2{-}q_2p_2)\otimes 1\\
			&\quad+x^2y\otimes(p_1q_1{+}q_1p_1{-}p_2q_2{-}q_2p_2),\\
			\Delta(p_1q_2{+}q_1p_2{+}p_2q_1{+}q_2p_1)&=(p_1q_2{+}q_1p_2{+}p_2q_1{+}q_2p_1)\otimes 1\\
			&\quad+\frac{1}{2}(1{+}x^2{-}x^2y{+}y)\otimes(p_1q_2{+}q_1p_2{+}p_2q_1{+}q_2p_1).\\
				\end{split}
		\end{equation*}
		Thus $p_1^2=\lambda(1-y)+\mu(1-x^2y)$, $p_2^2=\lambda(1-y)-\mu(1-x^2y)$, for some $\lambda$, $\mu$ $\in \mathbb{K}$, $p_1p_2+p_2p_1=0$.
		Since $x(p_1p_2+p_2p_1)=-(p_1p_2+p_2p_1)x$, relations (\ref{u21}---\ref{u24},) hold in $A$. Thus there is a surjective Hopf homomorphism from $\mathfrak{U_{2}}(J)$ to $A$. We can observe that all elements of $\mathfrak{U}_2(J)$ can be expressed by linear combinations of
			$$
		\{\,p_1^ip_2^jq_1^kq_2^lx^sy^mt^n\mid i,j,k,l,m,n\in \mathbb{I}_{0,1},s\in \mathbb{I}_{0,3}\,\}.
		$$
		In fact, by the Diamond Lemma, the set is a basis of $\mathfrak{U}_2(J)$. Then $A\cong \mathfrak{U}_{2}(J)$.
		
		For $(2)$,	suppose that $\Phi:\mathfrak{U}_{2}(I)\to \mathfrak{U}_{i}(I^{'})$ is an isomorphism of Hopf algebras, where $I^{'}=\{\lambda^{'},\mu^{'},\alpha^{'},\beta^{'},\gamma^{'},\eta^{'}\}$, and $\mathfrak{U}_{i}(I^{'})$ is generated by $x,y,t,p_1^{'},p_2^{'},q_1^{'},q_2^{'}$.When $\Phi|_H \in \{\tau_{1},\tau_{3},\tau_{17},\tau_{19},\tau_{37},\tau_{39},\tau_{53},\tau_{55}\}$, then there exist nonzero parameters $a_1,a_2,b_1,b_2$ such that
		\begin{gather}
			\Phi(p_1)=a_1p^{'}_1+a_2q_1^{'},\ \Phi(p_2)=a_1p^{'}_2+a_2q_2^{'},\\
			\Phi(q_1)=b_1p_1^{'}+b_2q_1^{'},\
			\Phi(q_2)=b_1p_2^{'}+b_2q_2^{'},
		\end{gather}
		then relation (\ref{(var7)}) holds, when $\Phi|_H \in \{\tau_{13},\tau_{14}\}$,
		\begin{gather}
			\Phi(p_1)=a_1p^{'}_1+a_2q_1^{'},\ \Phi(p_2)=\xi(a_1p^{'}_2+a_2q_2^{'}),\\
			\Phi(q_1)=b_1p_1^{'}+b_2q_1^{'},\
			\Phi(q_2)=\xi(b_1p_2^{'}+b_2q_2^{'}),
		\end{gather}
		then relation (\ref{(var8)}) holds.

		For $(3)$, suppose that $\Phi:\mathfrak{U}_{3}(I)\to \mathfrak{U}_{i}(I^{'})$ is an isomorphism of Hopf algebras, where $I^{'}=\{\lambda^{'},\mu^{'},\alpha^{'}\}$, and $\mathfrak{U}_{i}(I^{'})$ is generated by $x,y,t,p_1^{'},p_2^{'},q_1^{'},q_2^{'}$. When $\Phi|_H \in \{\tau_{1},\tau_{3},\tau_{17},\tau_{19},\tau_{33},\tau_{35}$, $\tau_{49},\tau_{51}\}$, then there exist nonzero parameters $a_1,a_2,b_1,b_2$ such that
		\begin{gather}
			\Phi(p_1)=a_1p^{'}_1+a_2q_1^{'},\ \Phi(p_2)=a_1p^{'}_2+a_2q_2^{'},\\
			\Phi(q_1)=b_1p_1^{'}+b_2q_1^{'},\
			\Phi(q_2)=b_1p_2^{'}+b_2q_2^{'},
		\end{gather}
		then relation (\ref{(var6)}) holds.\\
		For $(4)$, suppose that $\Phi:\mathfrak{U}_{6}(I)\to \mathfrak{U}_{i}(I^{'})$ is an isomorphism of Hopf algebras, where $I^{'}=\{\lambda^{'},\mu^{'}\}$, and $\mathfrak{U}_{i}(I^{'})$ is generated by $x,y,t,p_1^{'},p_2^{'},q_1^{'},q_2^{'}$. When $\Phi|_H \in \{\tau_{1},\tau_{3},\tau_{5},\tau_{7},\tau_{9},\tau_{10},\tau_{13}$, $\tau_{14}, \tau_{17},\tau_{19},\tau_{21},\tau_{23},\tau_{25},\tau_{26},\tau_{29},\tau_{30}\}$, then there exist nonzero parameters $a_1,a_2$ such that
		\begin{gather}
			\Phi(p_1)=a_1p^{'}_1,\ \Phi(p_2)=a_1p^{'}_2,\
			\Phi(q_1)=a_2q^{'}_1,\
			\Phi(q_2)=a_2q^{'}_2,
		\end{gather}
		then realtion  (\ref{(var4.1)}) holds,
		or there exist nonzero parameters $a_1,a_2$ such that
		\begin{gather}
			\Phi(p_1)=a_1q^{'}_1,\ \Phi(p_2)=a_1q^{'}_2,\
			\Phi(q_1)=a_2p^{'}_1,\
			\Phi(q_2)=a_2p^{'}_2,
		\end{gather}
		then realtion  (\ref{(var4.2)}) holds, and
		whenhen $\Phi|_H \in \{\tau_{2},\tau_{4},\tau_{6},\tau_{8},\tau_{11},\tau_{12},\tau_{15},\tau_{16},\tau_{18}, \tau_{20},\tau_{22},\tau_{24},\tau_{27}$, $\tau_{28},\tau_{31},\tau_{32}\}$, then there exist nonzero parameters $a_1,a_2$ such that
		\begin{gather}
			\Phi(p_1)=a_1p^{'}_2,\ \Phi(p_2)=-a_1p^{'}_1,\
			\Phi(q_1)=a_2q^{'}_2,\
			\Phi(q_2)=-a_2q^{'}_1,
		\end{gather}
		then realtion  (\ref{(var4.3)}) holds,
		or there exist nonzero parameters $a_1,a_2$ such that
		\begin{gather}
			\Phi(p_1)=a_1q^{'}_2,\ \Phi(p_2)=-a_1q^{'}_1,\
			\Phi(q_1)=a_2p^{'}_2,\
			\Phi(q_2)=-a_2p^{'}_1.
		\end{gather}
		then realtion  (\ref{(var4.4)}) holds.

		For $(5)$, suppose that $\Phi:\mathfrak{U}_{9}(I)\to \mathfrak{U}_{i}(I^{'})$ is an isomorphism of Hopf algebras, where $I^{'}=\{\lambda^{'},\mu^{'},\alpha^{'},\beta^{'}\}$, and $\mathfrak{U}_{i}(I^{'})$ is generated by $x, y, t, p_1^{'}, p_2^{'}, q_1^{'}, q_2^{'}$. When $\Phi|_H \in \{\tau_{1},\tau_{5},\tau_{33},\tau_{37}\}$, then there exist nonzero parameters $a_1,a_2,b_1,b_2$ such that
		\begin{gather}
			\Phi(p_1)=a_1p^{'}_1+a_2q_1^{'},\ \Phi(p_2)=a_1p^{'}_2+a_2q_2^{'},\\
			\Phi(q_1)=b_1p_1^{'}+b_2q_1^{'},\
			\Phi(q_2)=b_1p_2^{'}+b_2q_2^{'},
		\end{gather}
		then relation (\ref{(var2)}) holds; when $\Phi|_H \in \{\tau_{2},\tau_{6},\tau_{34},\tau_{38}\}$,
		\begin{gather}
			\Phi(p_1)=a_1p^{'}_1+a_2q_1^{'},\ \Phi(p_2)=-\xi(a_1p^{'}_2+a_2q_2^{'}),\\
			\Phi(q_1)=b_1p_1^{'}+b_2q_1^{'},\
			\Phi(q_2)=-\xi(b_1p_2^{'}+b_2q_2^{'}),
		\end{gather}
		then relation (\ref{(var3)}) holds; when $\Phi|_H \in \{\tau_{3},\tau_{7},\tau_{35},\tau_{39}\}$,
		\begin{gather}
			\Phi(p_1)=a_1p^{'}_1+a_2q_1^{'},\ \Phi(p_2)=-(a_1p^{'}_2+a_2q_2^{'}),\\
			\Phi(q_1)=b_1p_1^{'}+b_2q_1^{'},\
			\Phi(q_2)=-(b_1p_2^{'}+b_2q_2^{'}),
		\end{gather}
		then relation (\ref{(var4)}) holds;  when $\Phi|_H \in \{\tau_{4},\tau_{8},\tau_{36},\tau_{40}\}$,
		\begin{gather}
			\Phi(p_1)=a_1p^{'}_1+a_2q_1^{'},\ \Phi(p_2)=\xi(a_1p^{'}_2+a_2q_2^{'}),\\
			\Phi(q_1)=b_1p_1^{'}+b_2q_1^{'},\
			\Phi(q_2)=\xi(b_1p_2^{'}+b_2q_2^{'}).
		\end{gather}
		then relation (\ref{(var5)}) holds.

		We complete the proof.
	\end{proof}
	
	\begin{defi}\label{def15}
		Denote by $\mathfrak{U}_{13}(I)$ the algebra  generated by $x,y,t,p_{1},p_{2},q_{1},q_{2}$ satisfying relations (\ref{3.1}), and
		\begin{gather}
			xp_{1}=p_{1}x,\
			yp_{1}=-p_{1}y,\
			tp_{1}=\xi p_{1}x^2t,\
			xp_{2}=p_{2}x,\
			yp_{2}=-p_{2}y,\
			tp_{2}=-\xi p_{2}x^2t,\\
			xq_{1}=- q_{1}x,\
			yq_{1}=q_{1}y,\
			tq_{1}=q_{1}t,\
			xq_{2}=-q_{2}x,\
			yq_{2}=q_{2}y,\
			tq_{2}=-q_{2}t,\\
			p_{1}^2=p_2^2=\lambda(1-x^2),\
			p_{1}p_{2}+p_{2}p_1=0,\\	
			q_{1}^2=q_2^2=\mu(1-x^2),\
			q_{1}q_{2}+q_{2}q_1=\alpha(1-x^2).
		\end{gather}
		It is a Hopf algebra with its coalgebra structure determined by (\ref{3.2}) and
		\begin{equation*}
			\begin{split}
				\Delta(p_1)&=p_1\otimes 1+\frac{1}{2}(\,(1+x^2)y\otimes p_1- (1-x^2)yx\otimes p_2\,),\\
				\Delta(p_2)&=p_2\otimes 1+\frac{1}{2}(\, (1+x^2)y\otimes p_2-(1-x^2)y\otimes p_1\,),\\
				\Delta(q_1)&=q_1\otimes 1+\frac{1}{2}(\,(1+x^2)x\otimes q_1+\xi (1-x^2)x\otimes q_2\,),\\
				\Delta(q_2)&=q_2\otimes 1+\frac{1}{2}(\, (1+x^2)x\otimes q_2-\xi (1-x^2)x\otimes q_1\,).\\
			\end{split}
		\end{equation*}
	\end{defi}
	\begin{defi}\label{def16}
		Denote by $\mathfrak{U}_{14}(\lambda,\mu)$ the algebra  generated by $x,y,t,p_{1},p_{2},q_{1},q_{2}$ satisfying relations (\ref{3.1}), and
		\begin{gather}
			xp_{1}=p_{1}x,\
			yp_{1}=-p_{1}y,\
			tp_{1}=\xi p_{1}x^2t,\
			xp_{2}=p_{2}x,\
			yp_{2}=-p_{2}y,\
			tp_{2}=-\xi p_{2}x^2t,\\
			xq_{1}=- q_{1}x,\
			yq_{1}=q_{1}y,\
			tq_{1}=q_{1} t,\
			xq_{2}=- q_{2}x,\
			yq_{2}=q_{2}y,\
			tq_{2}=-q_{2}t,\\
			p_{1}^2=p_2^2=\lambda(1-x^2),\
			p_{1}p_{2}+p_{2}p_1=0,\\	
			q_{1}^2=q_2^2=\mu(1-x^2),
			q_{1}q_{2}+q_{2}q_1=0.
		\end{gather}
		It is a Hopf algebra with its coalgebra structure determined by (\ref{3.2}) and
		\begin{equation*}
			\begin{split}
				\Delta(p_1)&=p_1\otimes 1+\frac{1}{2}(\,(1+x^2)y\otimes p_1- (1-x^2)y\otimes p_2\,),\\
				\Delta(p_2)&=p_2\otimes 1+\frac{1}{2}(\, (1+x^2)y\otimes p_2-(1-x^2)y\otimes p_1\,),\\
				\Delta(q_1)&=q_1\otimes 1+\frac{1}{2}(\,(1+x^2)y\otimes q_1- (1-x^2)y\otimes q_2\,),\\
				\Delta(q_2)&=q_2\otimes 1+\frac{1}{2}(\, (1+x^2)y\otimes q_2-(1-x^2)y\otimes q_1\,).
			\end{split}
		\end{equation*}
	\end{defi}
	\begin{defi}\label{def17}
		Denote by $\mathfrak{U}_{15}(I)$ the algebra  generated by $x,y,t,p_{1},p_{2},q_{1},q_{2}$ satisfying relations (\ref{3.1}), and
		\begin{gather}
			xp_{1}=p_{1}x,\
			yp_{1}=p_{1}y,\
			tp_{1}=-p_{1}t,\
			xp_{2}=-p_{2}x,\
			yp_{2}=p_{2}y,\
			tp_{2}=p_{2}t,\\
			xq_{1}=q_{1}x,\
			yq_{1}=q_{1}y,\
			tq_{1}=-q_{1}t,\
			xq_{2}=-q_{2}x,\
			yq_{2}=q_{2}y,\
			tq_{2}=q_{2}t,\\
			p_{1}^2=p_2^2=\lambda(1-x^2y),\
			p_{1}p_{2}+p_{2}p_1=0,\\	
			q_{1}^2+q_2^2=\mu(1-y),\
			q_{1}q_{2}+q_{2}q_1=0,\\
			p_1q_1+q_1p_1+p_2q_2+q_2p_2=\alpha(1-x^2),\\
			p_1q_2+q_1p_2+p_2q_1+q_2p_1=0.
		\end{gather}
		It is a Hopf algebra with its coalgebra structure determined by (\ref{3.2})  and
		\begin{equation*}
			\begin{split}
				\Delta(p_1)&=p_1\otimes 1+\frac{1}{2}(\,(1+x^2)xt\otimes p_1+\xi (1-x^2)xyt\otimes p_2\,),\\
				\Delta(p_2)&=p_2\otimes 1+\frac{1}{2}(\,\xi (1-x^2)xyt\otimes p_1+(1+x^2)xyt\otimes p_2\,),\\
				\Delta(q_1)&=q_1\otimes 1+\frac{1}{2}(\,(1+x^2)xyt\otimes q_1- \xi(1-x^2)xt\otimes q_2\,),\\
				\Delta(q_2)&=q_2\otimes 1+\frac{1}{2}(\, (1+x^2)xt\otimes q_2-\xi(1-x^2)xt\otimes q_1\,).
			\end{split}
		\end{equation*}
	\end{defi}
	\begin{defi}\label{def18}
		Denote by $\mathfrak{U}_{16}(\lambda,\mu)$ the algebra  generated by $x,y,t,p_{1},p_{2},q_{1},q_{2}$ satisfying relations (\ref{3.1}), and
		\begin{gather}
			xp_{1}=p_{1}x,\
			yp_{1}=p_{1}y,\
			tp_{1}=- p_{1} t,\
			xp_{2}=-p_{2}x,\
			yp_{2}=p_{2}y,\
			tp_{2}=-p_{2}t,\\
			xq_{1}=q_{1}x,\
			yq_{1}=q_{1}y,\
			tq_{1}=-q_{1} t,\
			xq_{2}=- q_{2}x,\
			yq_{2}=q_{2}y,\
			tq_{2}=-q_{2}t,	\\
			p_{1}^2=\frac{1}{2}\lambda(1-y)+\frac{1}{2}\mu(1-x^2y),p_2^2=\lambda(1-y),\
			p_{1}p_{2}+p_{2}p_1=0,\\	
			q_{1}^2=q_2^2=\mu(1-y),\
			q_{1}q_{2}+q_{2}q_1=0,\\
			p_1q_1+q_1p_1+p_2q_2+q_2p_2=0,\\
			p_1q_2+q_1p_2+p_2q_1+q_2p_1=0.
		\end{gather}
		It is a Hopf algebra with its coalgebra structure determined by (\ref{3.2})  and
		\begin{equation*}
			\begin{split}
				\Delta(p_1)&=p_1\otimes 1+\frac{1}{2}(\,(1+x^2)t\otimes p_1+ (1-x^2)yt\otimes p_2\,),\\
				\Delta(p_2)&=p_2\otimes 1+\frac{1}{2}(\,(1+x^2)yt\otimes p_2-(1-x^2)yt\otimes p_1\,),\\
				\Delta(q_1)&=q_1\otimes 1+\frac{1}{2}(\,(1+x^2)yt\otimes q_1- (1-x^2)t\otimes q_2\,),\\
				\Delta(q_2)&=q_2\otimes 1+\frac{1}{2}(\, (1+x^2)t\otimes q_2+(1-x^2)t\otimes q_1\,).
			\end{split}
		\end{equation*}
	\end{defi}
	\begin{defi}\label{def19}
		Denote by $\mathfrak{U}_{17}(I,\beta)$ the algebra  generated by $x,y,t,p_{1},p_{2},q_{1},q_{2}$ satisfying relations (\ref{3.1}), and
		\begin{gather}
			xp_{1}=-p_{1}x,\
			yp_{1}=p_{1}y,\
			tp_{1}= p_{1}t,\
			xp_{2}=-p_{2}x,\
			yp_{2}=p_{2}y,\
			tp_{2}=- p_{2} x^2t,\\
			xq_{1}=- q_{1}x,\
			yq_{1}=q_{1}y,\
			tq_{1}= q_{1}t,\
			xq_{2}=-q_{2}x,\
			yq_{2}=q_{2}y,\
			tq_{2}= -q_{2} t,\\
			p_{1}^2=p_2^2=0,\
			p_{1}p_{2}+p_{2}p_1=\lambda(1-x^2),\\	
			q_{1}^2+q_2^2=0,\
			q_{1}q_{2}+q_{2}q_1=\mu(1-x^2),\\
			p_1q_1+q_1p_1+p_2q_2+q_2p_2=\alpha(1-x^2y),\\
			p_1q_2+q_1p_2+p_2q_1+q_2p_1=\beta(1-y).
		\end{gather}
		It is a Hopf algebra with its coalgebra structure determined by (\ref{3.2})  and
		\begin{equation*}
			\begin{split}
				\Delta(p_1)&=p_1\otimes 1+\frac{1}{2}(\,(1+x^2)x\otimes p_1- \xi (1-x^2)x\otimes p_2\,),\\
				\Delta(p_2)&=p_2\otimes 1+\frac{1}{2}(\,(1+x^2)x\otimes p_2-\xi(1-x^2)x\otimes p_1\,),\\
				\Delta(q_1)&=q_1\otimes 1+\frac{1}{2}(\,(1+x^2)xy\otimes q_1- \xi (1-x^2)xy\otimes q_2\,),\\
				\Delta(q_2)&=q_2\otimes 1+\frac{1}{2}(\,(1+x^2)xy\otimes q_2+\xi(1-x^2)xy\otimes q_1\,).
			\end{split}
		\end{equation*}
	\end{defi}
	
	\begin{defi}\label{def20}
		Denote by $\mathfrak{U}_{19}(I)$ the algebra generated by $x,y,t,p_{1},p_{2},q_{1},q_{2}$ satisfying relations (\ref{3.1}), and
		\begin{gather}
			xp_{1}=\xi p_{1}x,\quad
			yp_{1}=p_{1}y,\quad
			tp_{1}=- p_{2}t,\quad
			xp_{2}=-\xi p_{2}x,\quad
			yp_{2}=p_{2}y,\quad
			tp_{2}=p_{1}t,\\
			xq_{1}=\xi q_{1}x,\quad
			yq_{1}=q_{1}y,\quad
			tq_{1}= -q_{2} t,\quad
			xq_{2}=-\xi q_{2}x,\quad
			yq_{2}=q_{2}y,\quad
			tq_{2}= q_{1} t,	\\
			p_{1}^2=p_2^2=0,\quad
			p_{1}p_{2}+p_{2}p_1=\lambda(1-y),\\		q_{1}^2=q_2^2=0,\quad
			q_{1}q_{2}+q_{2}q_1=\mu(1-y),\\
			p_1q_1+q_1p_1+p_2q_2+q_2p_2=\alpha(1-y),\\
			p_1q_2+q_1p_2+p_2q_1+q_2p_1=0.
		\end{gather}
		It is a Hopf algebra with its coalgebra structure determined by (\ref{3.2})  and
		\begin{equation*}
			\begin{split}
				\Delta(p_1)&=p_1\otimes 1+x^2\otimes p_1,\\
				\Delta(p_2)&=p_2\otimes 1+x^2y\otimes p_2,\\
				\Delta(q_1)&=q_1\otimes 1+x^2y\otimes q_1,\\
				\Delta(q_2)&=q_2\otimes 1+x^2\otimes q_2.
			\end{split}
		\end{equation*}
	\end{defi}
	
	\begin{pro}
		$(1)$	Assume that  $A$ is a finite-dimensional Hopf algebra with the coradical $H$ such that  the infinitesimal braiding is isomorphic to $\Omega_{i}$ for some $i\in \{13$,  $14$,  $15$, $16$, $17$, $19\}$, then $A\cong \mathfrak{U}_{i}(I)$.
		
		$(2)$ For $i\in\{13,14,15,16,17,19\}$, $\mathfrak{U}_{i}(I)\cong \mathfrak{U}_{i}(I^{'})$ if and only if there exist nonzero parameters $a_1$, $a_2$ such that
		\begin{gather}\label{(var1)}
			a_1^2\lambda^{'}=\lambda,\ a_2^2\mu^{'}=\mu,\ a_2^2\alpha^{'}=\alpha,
		\end{gather}
		or satisfying
		\begin{gather}
			a_1^2\lambda^{'}=\lambda,\ a_2^2\mu^{'}=-\mu,\ a_2^2\alpha^{'}=-\alpha,
		\end{gather}
		or satisfying
		\begin{gather}
			a_1^2\lambda^{'}=-\lambda,\ a_2^2\mu^{'}=\mu,\ a_2^2\alpha^{'}=\alpha,
		\end{gather}
		or satisfying
		\begin{gather}
			a_1^2\lambda^{'}=-\lambda,\ a_2^2\mu^{'}=-\mu,\ a_2^2\alpha^{'}=-\alpha,
		\end{gather}
		or satisfying
		\begin{gather}
			a_1^2\lambda^{'}=\lambda,\ a_2^2\mu^{'}=\mu,\ a_2^2\alpha^{'}=-\alpha,
		\end{gather}
		or satisfying
		\begin{gather}
			a_1^2\lambda^{'}=-\lambda,\ a_2^2\mu^{'}=\mu,\ a_2^2\alpha^{'}=-\alpha,
		\end{gather}
		or satisfying
		\begin{gather}
			a_1^2\lambda^{'}=-\lambda,\ a_2^2\mu^{'}=-\mu,\ a_2^2\alpha^{'}=\alpha,
		\end{gather}
		or satisfying
		\begin{gather}
			a_1^2\lambda^{'}=\lambda,\ a_2^2\mu^{'}=-\mu,\ a_2^2\alpha^{'}=\alpha.
		\end{gather}
			\end{pro}
		\begin{proof}
			For $(1)$,  the proof is analogous to that of the previuos Propostion \ref{6u2}.
			
			For $(2)$, suppose that $\Phi:\mathfrak{U}_{13}(I)\to \mathfrak{U}_{i}(I^{'})$ is an isomorphism of Hopf algebras, where $I^{'}=\{\lambda^{'},\mu^{'},\alpha^{'}\}$, and $\mathfrak{U}_{i}(I^{'})$ is generated by $x,y,t,p_1^{'},p_2^{'},q_1^{'},q_2^{'}$. When $\Phi|_H \in \{\tau_{1},\tau_{3}\}$, then there exist nonzero parameters $a_1,a_2$ such that
			\begin{gather}
				\Phi(p_1)=a_1p^{'}_1,\ \Phi(p_2)=a_1p^{'}_2,\ \Phi(q_1)=a_2q^{'}_1,\  \Phi(q_2)=a_2q^{'}_2,
			\end{gather}
			thus relation (\ref{(var1)}) holds; when $\Phi|_H\in \{\tau_{13},\tau_{14}\}$,
			\begin{gather}
				\Phi(p_1)=a_1p^{'}_1,\ \Phi(p_2)=-a_1p^{'}_2,\ \Phi(q_1)=a_2q^{'}_1,\hspace{1em}\Phi(q_2)=a_2q^{'}_2,
			\end{gather}
			then  relation (\ref{(var1)}) holds.	
			
			For  the other cases, the proofs are completely analogous.
		\end{proof}
		
		\noindent
		\textbf{Proof of Theorem B}: Let $N$ be one of the Yetter-Drinfeld modules listed in
		$(1)$, $(7)$, $(8)$, $(9)$, $(10)$ of \textbf{Theorem A}. Let $A$ be a finite-dimensional Hopf algebra over $H$ such
		that its infinitesimal braiding is isomorphic to $N$. By Theorem 6.1, $gr \,A\cong \mathcal{B}(N)\sharp H$.
		Combining Propositions 6.12, 6.14, 6.20, 6.27 together, this completes the proof.\qed

	\section*{Acknowledgment}
	 Y.H. Wu gratefully acknowledges the financial support from East China Normal University, also thanks Prof. Yinhuo Zhang for his guidance and valuable comments during her visit to Hasselt University in Belgium. She is grateful to  Dr. Yiwei Zheng  for helpful comments and suggestions on the manuscript, and to Dr. Rongchuan Xiong for sharing relevant references, and to Dr. Yuxing Shi for providing GAP computations related to Nichols algebras. H.Y. Wang appreciates the support of the China Scholarship Council (CSC. No. 202506140042) for his visit to Université Paris Cité.  N.H. Hu is supported by the NNSF of China (Grant No. 12171155), and in part by
	the Science and Technology Commission of Shanghai Municipality (Grant No. 22DZ2229014).


\begin{thebibliography}{100}
		\bibitem{AS} N. Andruskiewitsch and H.-J. Schneider, \textit {Lifting of quantum linear spaces and pointed Hopf algebras of
			order $p^3$}, J. Algebra 209 (2) (1998), 658---691.
		
		\bibitem{AS0} N. Andruskiewitsch and H.-J. Schneider, \textit {Finite quantum groups and Cartan matrices}, Adv. Math. 154 (1) (2000), 1---45.
		
		
		\bibitem{AS2} N. Andruskiewitsch and H.-J. Schneider, \textit {Pointed Hopf algebras}, New directions in Hopf algebras, Math. Sci. Res. Inst. Publ., 43, Cambridge Univ. Press, Cambridge, 2002, 1---68.
		
		\bibitem{AS07} N. Andruskiewitsch and S. Zhang, \textit {On pointed hopf algebras associated to some conjugacy classes in $S_n$}, Proc. Amer. Math. Soc. 135 (2007), 2723-2731.
		
		
		
		\bibitem{AFG10} N. Andruskiewitsch, F. Fantino, M. Gra\~{n}a and L. Vendramin, \textit {Pointed Hopf algebras over some sporadic
			simple groups}, C. R. Math. Acad. Sci. Paris 348 (2010), 605---608.
		
		\bibitem{AHS} N. Andruskiewitsch, I. Heckenberger and H.-J. Schneider, \textit {The Nichols algebra of a semisimple Yetter-Drinfeld module}, Amer. J. Math. 132 (6) (2010), 1493---1547.
		
		\bibitem{AS3} N. Andruskiewitsch and H.-J. Schneider \textit {On the classification of finite-dimensional pointed Hopf algebras}, Ann. of Math. 171 (1) (2010), 375---417.
		
		\bibitem{AV} N. Andruskiewitsch and C. Vay, \textit {Finite dimensional Hopf algebras over the dual group algebra of the symmetric
			group in three letters}, Comm. Algebra 39 (12) (2011), 4507---4517.
		\bibitem{AFG11} N. Andruskiewitsch, F. Fantino, M. Gra\~{n}a and L. Vendramin, \textit {Finite-dimensional pointed Hopf algebras
			with alternating groups are trivial}, Ann. Mat. Pura Appl. (4) 190 (2) (2011), 225---245.
		
		
		\bibitem{AI} N. Andruskiewitsch, I. Angiono, A. G. Iglesias, A. Masuoka and C. Vay, \textit {Lifting via cocycle deformation},
		J. Pure Appl. Algebra 218 (4) (2014), 684---703.
		
		\bibitem{ACG15a} N. Andruskiewitsch, G. Carnovale and G. A. Garc\'{i}a, \textit {Finite-dimensional pointed Hopf algebras over finite
			simple groups of Lie type I. Non-semisimple classes in $\text{PSL}_n(q)$}, J. Algebra 442 (2015), 36---65.
		
		\bibitem{ACG15b} N. Andruskiewitsch, G. Carnovale and G. A. Garc\'{i}a, \textit {Finite-dimensional pointed Hopf algebras over finite
			simple groups of Lie type II: unipotent classes in symplectic groups}, Commun. Contemp. Math.
		18 (4) (2016), 1550053, 35.
		
		\bibitem{AAI} N. Andruskiewitsch, I. Angiono and A. G. Iglesias, \textit {Liftings of Nichols algebras of diagonal type I. Cartan
			type $A$}, Int. Math. Res. Not. (9) (2017), 2793---2884.
		\bibitem{AGM15} N. Andruskiewitsch, C. Galindo and M. M\"{u}ller, \textit {Examples of finite-dimensional Hopf algebras with the
			dual Chevalley property}, Publ. Mat., 61 (2) (2017): 445---474.
		
		\bibitem{AGi2018} N. Andruskiewitsch and J. M. J. Giraldi, \textit {Nichols algebras that are quantum planes}, Linear and Multilinear Algebra 66 (5) (2018), 961---991.
		
		
		
		\bibitem{AHV24} N. Andruskiewitsch, I. Heckenberger, and L. Vendramin, \textit {Pointed Hopf algebras of odd dimension and Nichols algebras over solvable groups}, arXiv.2411.02304.
		
		\bibitem{G} G. M. Bergman, \textit {The diamond lemma for ring theory}, Adv. in Math. 29 (2) (1978), 178---218.
		
		
		\bibitem{CS} C. C\v{a}linescu, S. D\v{a}sc\v{a}lescu, A. Masuoka and C. Menini, \textit {Quantum lines over non-cocommutative
			cosemisimple Hopf algebras}, J. Algebra 273 (2) (2004), 53---779.
		
		
		\bibitem{FG} F. Fantino and G. A. Garc\'{i}a, \textit {On pointed Hopf algebras over dihedral groups}, Pacific J. Math. 252 (1) (2011), 69---91.
		
		\bibitem{X} X. Fang, \textit {On defining ideals and differential algebras of Nichols algebras}, J. Algebra 346 (2011), 299---331.
		
		\bibitem{GG16} G. A. Garc\'{i}a and J. M. J. Giraldi, \textit {On Hopf algebras over quantum subgroups}, J. Pure Appl. Algebra 223 (2) (2019), 738---768.
		
		
		
		\bibitem{GGI} G. A. Garc\'{i}a and A. G. Iglesias, \textit {Finite-dimensional pointed Hopf algebras over $\mathbb{S}_4$},
		Israel J. Math. 183 (2011), 417---444.
		
		\bibitem{Gn00} M. Gra\~{n}a, \textit {A freeness theorem for Nichols algebras}, J. Algebra 231 (1) (2000), 235---257.
		
		\bibitem{H06} I. Heckenberger, \textit {The Weyl groupoid of a Nichols algebra of diagonal type}, Invent. Math. 164 (2006), 175---188.
		
		\bibitem{H09} I. Heckenberger, \textit {Classification of arithmetic root systems}, Adv. Math. 220 (1) (2009), 59---124.
		
		\bibitem{HS}   I. Heckenberger and H.-J. Schneider, \textit {Nichols algebras over groups with finite root system of rank two I}, J. Algebra, 324 (11) (2010), 3090---3114.
		
		
		\bibitem{HEL24}  I. Heckenberger, E. Meir and L. Vendramin, \textit {Finite-dimensional Nichols algebras of simple Yetter-Drinfeld modules (over groups) of prime dimension}, Adv.  Math.
		444 (2024), 109637, 30.
		
		\bibitem{HX16} N. H. Hu and R. C. Xiong, \textit {Some Hopf algebras of dimension $72$ without the Chevalley property}, arXiv:1612.0498.
		
		\bibitem{HX18} N. H. Hu and R. C. Xiong, \textit {On families of Hopf algebras without the dual Chevalley property}, Rev. Un. Mat. Argentina 59 (2) (2018), 443---469.
		
		
		
		\bibitem{GIV} A. G. Iglesias and C. Vay, \textit {Finite-dimensional pointed or copointed Hopf algebras over affine racks}, J.
		Algebra 397 (2014), 379---406.
		
		\bibitem{K} I. Kaplansky, \textit {Bialgebras}, Lecture Notes in Mathematics. Department of Mathematics, University of
		Chicago, Chicago, Ill., 1975.
		
		\bibitem{K00} Y. Kashina, \textit {Classification of semisimple Hopf algebras of dimension $16$}, J. Algebra 232 (2) (2000), 617---663.
		
		
		\bibitem{CK} C. Kassel, \textit{Quantum Groups}, Graduate Texts in Mathematics,
		Vol. 155, Springer-Verlag, New York/Berlin, 1995.
		
		
		\bibitem{M} S. Montgomery, \textit {Hopf Algebras and their Actions on Rings}, CMBS Reg. Conf. Ser. in Math. 82, Amer. Math. Soc., 1993.
		
		
		\bibitem{R85} D. E. Radford, \textit {The structure of Hopf algebras with a projection}, J. Algebra 92 (2) (1985), 322---347.
		
		\bibitem{R}D. E. Radford, \textit {Hopf Algebras}, Knots and Everything 49, World Scientific, 2012.
		
		\bibitem{S16} Y. X. Shi, \textit {Finite-dimensional Hopf algebras over the Kac-Paljutkin algebra $H_8$}, Rev. Un. Mat. Argentina 60 (1) (2019), 265---298.
		
		\bibitem{X19} R. C. Xiong, \textit{On Hopf algebras over the unique 12-dimensional Hopf algebra without the dual Chevalley property}, Comm. Algebra 47 (2019), 1516---1540.
		
		\bibitem{X192} R. C. Xiong, \textit{Some classification results on finite-dimensional Hopf algebras}, Ph. D. Dissertation of East China Normal University (pp. 347), May, 2019.
		
		
		\bibitem{X21} R. C. Xiong, N. H. Hu, \textit{Classification of finite-dimensional Hopf algebras over dual Radford algebras}, Bull. Belg. Math. Soc. Simon Stevin 28 (5) (2021), 633---688.
		\bibitem{X23} R. C. Xiong, \textit{On Hopf algebras over basic Hopf algebras of dimension 24}, to appear in Rev. Un. Mat. Argentina. 65 (2) (2023), 469---493.
		
		
		\bibitem{X232} R. C. Xiong,  \textit{Pointed Hopf algebras of dimension $p^2q$ in characteristic $p$}, J. Algebra 631 (2023), 355---400.
		
		
		\bibitem{X24} R. C. Xiong, \textit{Finite-dimensional Hopf algebras over the smallest non-pointed basic Hopf algera}, Front. Math. 19 (5) (2024), 865---890.
		
		
		
		\bibitem{SW} M. Sweedler, \textit {Hopf Algebras}, Benjamin, New York, 1969.
		
		\bibitem{Z211} Y. W. Zheng, Y. Gao and N. H. Hu, \textit{Finite-dimensional Hopf algebras over the Hopf algebra $H_{b:1}$ of Kashina}, J. Algebra 567 (1) (2021), 613---659.
		
		
		\bibitem{Z212} Y. W. Zheng, Y. Gao and N. H. Hu,  \textit{Finite-dimensional Hopf algebras over the Hopf algebra $H_{d:-1,1}$ of Kashina}, J. Pure Appl. Algebra 225 (4) (2021), 106527.
		
		\bibitem{Z23} Y. W. Zheng, Y. Gao,  N. H. Hu and  Y. X. Shi, \textit{On some classification of finite-dimensional Hopf algebras over the Hopf algebras $H_{b:1*}$} of Kashina, Commu. Algebra 51 (2023), 350---371.
		
		
		\bibitem{Z25} Y. W. Zheng, X. L. Yu, \textit{Some classifications of finite-dimensional Hopf algebras over $H_{a:y}$ of Kashina}, J. Algebra Appl. 24 (11) (2025), 2550257.
		
		
	\end{thebibliography}
\end{document}